\documentclass{tac}  
\usepackage[all,2cell]{xy}
\usepackage{mathrsfs,amssymb,stmaryrd,bbm,mathtools}

\newtheorem{theo}{Theorem}[section]
\newtheorem{lem}[theo]{Lemma}
\newtheorem{prop}[theo]{Proposition}
\newtheorem{coro}[theo]{Corollary}
\newtheoremrm{defi}[theo]{Definition}
\newtheoremrm{rem}[theo]{Remark}

\def\mathrmdef#1{\expandafter\def\csname#1\endcsname{{\rm#1}}}
\def\mathsfdef#1{\expandafter\def\csname#1\endcsname{{\rm\sf#1}}}
\def\mathcaldef#1{\expandafter\def\csname#1\endcsname{{\mathcal#1}}}
\mathrmdef{id}\mathrmdef{op}
\mathsfdef{Alg} \mathrmdef{Ps} \mathsfdef{Set} \mathsfdef{CAT}\mathcaldef{Ran}\mathcaldef{Lan}\mathrmdef{bi}
\mathrmdef{j}\mathrmdef{h}\mathrmdef{obj}\mathsfdef{CoAlg}\mathrmdef{Id}

\def\xxxx{\mathfrak{x}}
\def\yyyy{\mathfrak{y}}
\def\jjjj{\mathfrak{j}}
\def\eeee{\mathfrak{e}}
\def\llll{\mathfrak{l}}
\def\rrrr{\mathfrak{r}}
\def\uuuu{\mathfrak{u}}
\def\tttt{\mathfrak{t}}

\def\aaa{\mathfrak{A}}
\def\bbb{\mathfrak{B}}
\def\ccc{\mathfrak{C}}

\def\sss{\mathfrak{S}}

\def\AA{\mathbbmss{A}}
\def\BB{\mathbbmss{B}}
\def\CC{\mathbbmss{C}}

\def\AAA{\mathbb{A}}
\def\BBB{\mathbb{B}}
\def\CCC{\mathbb{C}}

\def\DDDD{\mathcal{D}}
\def\AAAA{\mathcal{A}}
\def\BBBB{\mathcal{B}}
\def\WWWW{\mathcal{W}}
\def\VVVV{\mathcal{V}}

\def\YYYY{\mathcal{Y}}

\def\TTTTT{\mathcal{T}}
\def\SSSSS{\mathcal{S}}

\def\KKKK{\mathcal{K}}

\begin{document}  
\title{On Biadjoint Triangles} 
\author{Fernando Lucatelli Nunes}
\address{CMUC, Department of Mathematics, University of Coimbra, 3001-501 Coimbra, Portugal}
\eaddress{lucatellinunes@student.uc.pt}


\thanks{This work was supported by CNPq, National Council for Scientific and Technological Development -- Brazil (245328/2012-2),  and by the Centre for Mathematics of the
University of Coimbra -- UID/MAT/00324/2013, funded by the Portuguese
Government through FCT/MCTES and co-funded by the European Regional Development Fund through the Partnership Agreement PT2020.}


\maketitle 
\begin{abstract}
We prove a biadjoint triangle theorem and its strict version, which are $2$-dimensional analogues of the adjoint triangle theorem of Dubuc. Similarly to the $1$-dimensional case, we demonstrate how we can apply our results to get the pseudomonadicity characterization (due to Le Creurer, Marmolejo and Vitale).

Furthermore, we study applications of our main theorems in the context of the $2$-monadic approach to coherence. As a direct consequence of our strict biadjoint triangle theorem, we give the construction (due to Lack) of the left $2$-adjoint to the inclusion of the strict algebras into the pseudoalgebras.

In the last section, we give two brief applications on lifting biadjunctions and pseudo-Kan extensions.
\end{abstract}

\section*{Introduction}
Assume that  $E:\AA\to\CC$, $J: \AA\to\BB $, $L:\BB\to \CC$ are functors such that there is a natural isomorphism
\[\xymatrix{  \AA\ar[rr]^-{J}\ar[dr]_-{E}&&\BB\ar[dl]^-{L}\\
&\CC\ar@{}[u]|-\cong & }\] 
Dubuc~\cite{Dubuc} proved that if $L: \BB\to\CC $ is precomonadic, $E: \AA\to \CC$ has a right adjoint and $\AA$ has some needed equalizers, then $J$ has a right adjoint. In this paper, we give a $2$-dimensional version of this theorem, called the biadjoint triangle theorem. More precisely, let $\aaa $, $\bbb $ and $\ccc $ be $2$-categories and assume that $$E:\aaa\to\ccc, J: \aaa\to\bbb, L:\bbb\to \ccc$$ are pseudofunctors such that $L$ is pseudoprecomonadic and $E$ has a right biadjoint. We prove that, if we have the pseudonatural equivalence below, then $J$ has a right biadjoint $G$, provided that $\aaa$ has some needed descent objects.
\[\xymatrix{  \aaa\ar[rr]^-{J}\ar[dr]_-{E}&&\bbb\ar[dl]^-{L}\\
&\ccc\ar@{}[u]|-\simeq & }\] 
We also give sufficient conditions under which the unit and the counit of the obtained biadjunction are pseudonatural equivalences, provided that $E$ and $L$ induce the same pseudocomonad. Moreover, we prove a strict version of our main theorem on biadjoint triangles. That is to say, we show that, under suitable conditions, it is possible to construct (strict) right 2-adjoints.

Similarly to the $1$-dimensional case~\cite{Dubuc}, the biadjoint triangle theorem can be applied to get the pseudo(co)monadicity theorem due to Le Creurer, Marmolejo and Vitale~\cite{CME}. Also, some of the constructions of biadjunctions related to two-dimensional monad theory given by Blackwell, Kelly and Power~\cite{Power} are particular cases of the biadjoint triangle theorem.

Furthermore, Lack~\cite{SLACK2002} proved what may be called a general coherence result: his theorem states that the inclusion of the strict algebras into the pseudoalgebras of a given $2$-monad $\TTTTT $ on a $2$-category $\ccc $ has a left $2$-adjoint and the unit of this $2$-adjunction is a pseudonatural equivalence, provided that $\ccc $ has and $\TTTTT $ preserves strict codescent objects. This coherence result is also a consequence of the biadjoint triangle theorems proved in Section \ref{Main}.

Actually, although the motivation and ideas of the biadjoint triangle theorems came from the original adjoint triangle theorem~\cite{Dubuc, RS2010} and its enriched version stated in Section \ref{Dubuc}, Theorem \ref{MAIN} may be seen as a generalization of the construction,  given in \cite{SLACK2002}, of the right biadjoint to the inclusion of the $2$-category of strict coalgebras into the $2$-category of pseudocoalgebras.

In Section \ref{Dubuc}, we give a slight generalization of Dubuc's theorem, in its enriched version (Proposition \ref{D1}). This version gives the $2$-adjoint triangle theorem for $2$-pre(co)monadicity, but it lacks applicability for biadjoint triangles and pseudopre(co)monadicity. Then, in Section \ref{Bilimits} we change our setting: we recall some definitions and results of the tricategory $2$-$\CAT $ of $2$-categories, pseudofunctors, pseudonatural transformations and modifications. Most of them can be found in Street's articles~\cite{RS80, RS87}.

Section \ref{Descent} gives definitions and results related to descent objects~\cite{RS80, RS87}, which is a very important type of $2$-categorical limit in $2$-dimensional universal algebra. Within our established setting, in Section \ref{Main} we prove our main theorems (Theorem \ref{MAIN} and Theorem \ref{MAINstrict}) on biadjoint triangles, while, in Section \ref{PRECOMONADIC}, we give consequences of such results in terms of pseudoprecomonadicity (Corollary \ref{PMAIN}), using the characterization of pseudoprecomonadic pseudofunctors given by Proposition \ref{prepseudocomonadicfundamental}, that is to say, Corollary \ref{preprepseudocomonadic}.

In Section \ref{counitunit}, we give results (Theorem \ref{unitunit} and Theorem \ref{counit}) on the counit and unit of the obtained biadjunction $J\dashv G$ in the context of biadjoint triangles, provided that $E$ and $L $ induce the same pseudocomonad. Moreover, we demonstrate the pseudoprecomonadicity characterization of \cite{CME} as a consequence of our Corollary \ref{preprepseudocomonadic}.

In Section \ref{CMEE}, we show how we can apply our main theorem to get the pseudocomonadicity characterization~\cite{HER, CME} and we give a corollary of Theorem \ref{counit} on the counit of the biadjunction $J\dashv G $ in this context. Furthermore, in Section \ref{coherence} we show that the theorem of \cite{SLACK2002} on the inclusion $\TTTTT\textrm{-}\CoAlg _{{}_{\mathsf{s}}}\to \mathsf{Ps}\textrm{-}\TTTTT\textrm{-}\CoAlg $ is a direct consequence of the theorems presented herein, giving a brief discussion on consequences of the biadjoint triangle theorems in the context of the $2$-(co)monadic approach to coherence. Finally, we discuss a straightforward application on lifting biadjunctions in Section \ref{App}.

Since our main application in Section \ref{App} is about construction of right biadjoints, we prove theorems for pseudoprecomonadic functors instead of proving theorems on pseudopremonadic functors. But, for instance, to apply the results of this work in the original setting of~\cite{Power}, or to get the construction of the left biadjoint given in~\cite{SLACK2002}, we should, of course, consider the dual version: the Biadjoint Triangle Theorem \ref{dualMAIN}.

I wish to thank my supervisor Maria Manuel Clementino for her support, attention and useful feedback during the preparation of this work, realized in the course of my PhD program at University of Coimbra.

\section{Enriched Adjoint Triangles}\label{Dubuc}
Consider a cocomplete, complete and symmetric monoidal closed category $V$. Assume that $L: \BBB\to \CCC $ is a $V$-functor and $(L\dashv U, \eta , \varepsilon )$ is a $V$-adjunction. We denote by $$\chi : \CCC(L-, -)\cong \BBB (-,U-)$$ its associated $V$-natural isomorphism, that is to say, for every object $X$ of $\BBB$ and every object $Z$ of $\CCC$, $\chi _ {{}_{(X,Z)}} = \BBB (\eta _ {{}_X}, UZ )\circ U_{{}_{LX,Z}} $.

\begin{prop}[Enriched Adjoint Triangle Theorem]\label{D1}
Let $(L\dashv U, \eta , \varepsilon) $, $(E\dashv R, \rho , \mu ) $ be $V$-adjunctions such that 
$$\xymatrix{  \AAA\ar[rr]^-{J}\ar[dr]_-{E}&&\BBB\ar[dl]^-{L}\\
&\CCC & }$$
is a commutative triangle of $V$-functors. Assume that, for each pair of objects $(A\in\AAA , Y\in \BBB ) $, the induced diagram 
$$\xymatrix{ \BBB (JA, Y) \ar[r]^-{L_{{}_{JA,Y}}} & \CCC (EA, LY)\ar@<-1ex>[rrr]_-{\CCC (EA, L(\eta _{{}_Y} ))}\ar@<1ex>[rrr]^-{L_{{}_{{}_{JA,ULY}}}\circ\hspace{0.2em} \chi_ {{}_{{}_{(JA,LY)}}}} &&& \CCC (EA, LULY ) }$$
is an equalizer in $V$. The $V$-functor $J$ has a right $V$-adjoint $G$ if and only if, for each object $Y$ of $\BBB$, the $V$-equalizer of
$$\xymatrix{ RLY \ar@<0.9 ex>[rrrr]^-{RL(U(\mu _{{}_{LY}})\eta _{{}_{JRLY}})\rho _ {{}_{RLY}}}\ar@<-0.9ex>[rrrr]_-{RL(\eta _ {{}_{Y}} )} &&&& RLULY}$$
exists in the $V$-category $\AAA $. In this case, this equalizer gives the value of $GY$.
\end{prop}   
\begin{proof}
For each pair of objects $(A\in\AAA, Y\in\BBB ) $, the $V$-natural isomorphism  $\CCC (E-, -)\cong \AAA (-, R-)$ gives the components of the natural isomorphism
\small
$$\xymatrix{ \BBB(JA, Y)\ar@{=}[dd] \ar[r]^-{L_{{}_{JA,Y}}} & \CCC (EA, LY)\ar[dd]|-{\cong }\ar@<-1ex>[rrr]_-{\CCC (EA, L(\eta _{{}_Y} ))}\ar@<1ex>[rrr]^-{L_{{}_{{}_{JA,ULY}}}\circ\hspace{0.2em} \chi_ {{}_{{}_{(JA,LY)}}}} &&& \CCC (EA, LULY ) \ar[dd]|-{\cong }\\
&&&&\\
\BBB(JA, Y) \ar[r] & \AAA(A, RLY) \ar@<1ex>[rrr]^{\AAA(A, r_ Y)}\ar@<-1ex>[rrr]_{\AAA(A, q_Y) } &&& \AAA(A, RLULY)}$$
\normalsize
in which $q_Y = RL(\eta _ {{}_{Y}} ) $ and $ r_Y=RL(U(\mu _{{}_{LY}})\eta _{{}_{JRLY}})\rho _ {{}_{RLY}} $. Thereby, since, by hypothesis, the top row is an equalizer, $\BBB(JA, Y)$ is the equalizer of $(\AAA(A, q_Y), \AAA (A, r_Y)) $.

Assuming that the pair $(q_ Y,r_ Y)$ has a $V$-equalizer $GY$ in $\AAA $ for every $Y$ of $\BBB $, we have that $\AAA(A, GY) $ is also an equalizer of $(\AAA(A, q_Y), \AAA (A, r_Y) )$. Therefore we get a $V$-natural isomorphism $\AAA ( -, GY)\cong \BBB(J-, Y)$.

Reciprocally, if $G$ is right $V$-adjoint to $J$, since $\AAA(-, GY)\cong \BBB(J-, Y) $ is an equalizer of $\left(\AAA(-, q_Y), \AAA(-, r_Y)\right) $, $GY$ is the $V$-equalizer of $(q_Y, r_Y) $.
This completes the proof that the $V$-equalizers of $q_Y, r_Y $ are also necessary.             
\end{proof}

The results on (co)monadicity in $V$-$\CAT $ are similar to those of the classical context of $\CAT $ (see, for instance, \cite{Dubuc2, RS72}). Actually, some of those results of the enriched context can be seen as consequences of the classical theorems because of Street's work \cite{RS72}. 

Our main interest is in Beck's theorem for $V$-precomonadicity. More precisely, it is known that the $2$-category $V$-$\CAT $ admits construction of coalgebras~\cite{RS72}. Therefore every left $V$-adjoint $L: \BBB \to \CCC $ comes with the corresponding Eilenberg-Moore factorization.
\small
$$\xymatrix{\BBB \ar[r]^-{\phi }\ar[rd]_-{L}& \CoAlg \ar[d]\\
&\CCC } $$
\normalsize 
If $V= \Set$, Beck's theorem asserts that $ \phi $ is fully faithful if and only if the diagram below is an equalizer for every object $Y$ of $\BBB $. In this case, we say that $L$ is precomonadic.
$$\xymatrix{  Y \ar[r]^-{\eta _ {{}_Y}} & ULY \ar@<-0.9 ex>[rr]_-{UL(\eta_ {{}_{Y}})}\ar@<0.9ex>[rr]^-{\eta _ {{}_{ULY}}} && ULULY}$$
With due adaptations, this theorem also holds for enriched categories. That is to say, $ \phi $ is $V$-fully faithful if and only if the diagram above is a $V$-equalizer for every object $Y$ of $\BBB $. This result gives what we need to prove Corollary \ref{DD1}, which is the enriched version for Dubuc's theorem~\cite{Dubuc}.

\begin{coro}\label{DD1}
Let $(L\dashv U, \eta , \varepsilon) $, $(E\dashv R, \rho , \mu ) $ be $V$-adjunctions and $J$ be a $V$-functor such that 
$$\xymatrix{  \AAA\ar[rr]^-{J}\ar[dr]_-{E}&&\BBB\ar[dl]^-{L}\\
&\CCC & }$$
\normalsize
commutes and $L$ is $V$-precomonadic. The $V$-functor $J$ has a right $V$-adjoint $G$ if and only if, for each object $Y$ of $\BBB$, the $V$-equalizer of
\small
$$\xymatrix{ RLY \ar@<0.9 ex>[rrrr]^-{RL(U(\mu _{{}_{LY}})\eta _{{}_{JRLY}})\rho _ {{}_{RLY}}}\ar@<-0.9ex>[rrrr]_-{RL(\eta _ {{}_{Y}} )} &&&& RLULY}$$
\normalsize
exists in the $V$-category $\AAA $. In this case, these equalizers give the value of the right adjoint $G$.
\end{coro}

\begin{proof}
The isomorphisms induced by the $V$-natural isomorphism 
$\chi :\CCC(L-,-)\cong\BBB(-,U-)$ are the components of the natural isomorphism
\small
$$\xymatrix{ \BBB (JA, Y)\ar@{=}[dd] \ar[rr]^-{L_{{}_{JA,Y}}} && \CCC (EA, LY)\ar[dd]|-{\chi_ {{}_{{}_{(JA,LY)}}}}\ar@<-1ex>[rrr]_-{\CCC (EA, L(\eta _{{}_Y} ))}\ar@<1ex>[rrr]^-{L_{{}_{{}_{JA,ULY}}}\circ\hspace{0.2em} \chi_ {{}_{{}_{(JA,LY)}}}} &&& \CCC (EA, LULY )\ar[dd]|-{\chi_ {{}_{{}_{(JA,LULY)}}}}\\
&& &&&\\
\BBB(JA, Y) \ar[rr]^-{\BBB (JA, \eta _ {{}_{Y}})} && \BBB(JA, ULY) \ar@<0.9 ex>[rrr]^-{\BBB(JA, \eta _ {{}_{ULY}})}\ar@<-0.9ex>[rrr]_-{\BBB(JA, UL(\eta_ {{}_{Y}}))} &&& \BBB(JA, ULULY)
}$$
\normalsize
Since $L$ is $V$-precomonadic, by the previous observations, the top row of the diagram above is an equalizer. Thereby, for every object $A$ of $\AAA $ and every object $Y$ of $\BBB $, the bottom row, which is the diagram $D^{JA}_Y $,
is an equalizer.  By Proposition \ref{D1}, this completes the proof. 
\end{proof}

Proposition \ref{D1} applies to the case of $\CAT$-enriched category theory. But it does not give results about pseudomonad theory. For instance, the construction above does not give the right biadjoint constructed in~\cite{Power, SLACK2002} 
\begin{center}
$\mathsf{Ps}\textrm{-}\TTTTT\textrm{-}\CoAlg\to \TTTTT\textrm{-}\CoAlg_{{}_{\mathsf{s}}} . $
\end{center}

Thereby, to study pseudomonad theory properly, we study biadjoint triangles, which cannot be dealt with only $\CAT$-enriched category theory. Yet, a $2$-dimensional version of the perspective given by Proposition \ref{D1} is what enables us to give the construction of (strict) right $2$-adjoint functors in Subsection \ref{StrictVersion}.

\section{Bilimits}\label{Bilimits}
We denote by $2\textrm{-}\CAT $ the tricategory of $2$-categories, pseudofunctors (homomorphisms), pseudonatural transformations (strong transformations) and modifications. Since this is our main setting, we recall some results and concepts related to $2\textrm{-}\CAT$. Most of them can be found in \cite{RS80}, and a few of them are direct consequences of results given there.

Firstly, to fix notation, we set the tricategory $2$-$\CAT$, defining pseudofunctors, pseudonatural transformations and modifications. Henceforth, in a given $2$-category, we always denote by $\cdot $ the vertical composition of $2$-cells and by $\ast $ their horizontal composition.

\begin{defi}[Pseudofunctor]
Let $\bbb, \ccc $ be $2$-categories. A \textit{pseudofunctor} $L:\bbb\to \ccc $ is a pair $(L , \llll ) $ with the following data:
\begin{itemize}
\item Function $L : \obj (\bbb )\to \obj (\ccc ) $;
\item For each pair $(X,Y)$ of objects in $\bbb $, functors $L _{{}_{X,Y}} : \bbb (X,Y)\to \ccc (LX, LY) $;
\item For each pair $g: X\to Y , h: Y\to Z $ of $1$-cells in $\bbb $, an invertible $2$-cell of $\ccc$: $$\llll _ {{}_{hg}}: L(h) L(g)\Rightarrow L(hg); $$
\item For each object $X$ of $\bbb $, an invertible $2$-cell in $\ccc$: $$\llll_ {{}_{X}}: \id _{{}_{LX}}\Rightarrow L(\id _ {{}_X} );$$
\end{itemize} 
such that, if $\hat{g}, g: X\to Y,  \hat{h}, h:Y\to Z,  f:W\to X  $ are $1$-cells of $\bbb $, and $\xxxx : g\Rightarrow \hat{g}, \yyyy : h\Rightarrow \hat{h} $ are $2$-cells of $\bbb $, the following equations hold:
\begin{enumerate}
\item Associativity:
\tiny
$$\xymatrix{ 
LW\ar[rr]^{L(f)}\ar[dd]_{L(hgf)}\ar[ddrr]|{L(gf)}
&&
LX\ar[dd]^{L(g)}\ar@{}[dl]|{\xLeftarrow{\llll _ {{}_{gf}}}} 
&&
LW\ar[rr] ^{L(f) }\ar[dd]_{L(hgf)}\ar@{}[dr]|{\xLeftarrow{\llll_ {{}_{(hg)f}}}}
&&
LX\ar[dd]^{L(g)}\ar[ddll]|{L(hg)}
\\
&&&=&&&
\\
LZ\ar@{}[ru]|{\xLeftarrow{\llll _ {{}_{h(gf)}}}}
&&
LY\ar[ll]^{L(h)}
&&
LZ 
&&
LY\ar[ll]^{L(h)} \ar@{}[ul]|{\xLeftarrow{\llll _ {{}_{hg}}}} 							
}$$
\normalsize
\item Identity:
\tiny
$$\xymatrix{  LW     \ar[rr]^{L(f)}\ar[dd]_{L(\id _ {{}_X}f)}&&
LX\ar@/_7ex/[dd]|{L(\id _ {{}_X}) }
                    \ar@{}[dd]|{\xLeftarrow{\llll _ {{}_X}} }
                    \ar@/^7ex/[dd]|{\id _ {{}_{LX}} }
&&
LW\ar[dd]_{L(f\id _ {{}_W})}
&& 
LW\ar@{=}[ll]\ar@/_7ex/[dd]|{L(\id _ {{}_W}) }
                    \ar@{}[dd]|{\xLeftarrow{\llll _ {{}_W}} }
                    \ar@/^7ex/[dd]|{\id _ {{}_{LW}} }
										&&
LW\ar@/_4ex/[dd]|{L(f) }
                    \ar@{}[dd]|{=}
                    \ar@/^4ex/[dd]|{L(f) }										
\\
&\ar@{}[l]|{\xLeftarrow{\llll _ {{}_{\id _ {{}_X} f }}}} &
&=&
&\ar@{}[l]|{\xLeftarrow{\llll _ {{}_{f\id _ {{}_W}  }}}} &
&=& \\										
LX\ar@{=}[rr]&&
LX
&&
LX &&LW\ar[ll]^{L(f)}																				
&& LX							}$$ 
\normalsize
\item Naturality:						
\tiny
$$\xymatrix{  LX\ar[dddd]_{L(\hat{h}\hat{g})}\ar@{=}[r] &LX     \ar@{=}[r] 
                    \ar[dd]_{L(\hat{g})} 
                       & 
              {LX}  \ar[dd]^{L(g) }      && LX\ar[dddd]_{L(\hat{h}\hat{g})}\ar@{=}[r] &LX     \ar[rr]^{L(g)} 
                    \ar[dddd]^{L(hg)} 
                       && 
              {LY}  \ar[dddd]_{L(h)}\\
							&&{}\ar@{}[l]|{\xLeftarrow{L(\xxxx )}} && &&&
															\\
              &{LY}\ar@{}[l]|{\xLeftarrow{\llll _ {{}_{\hat{h}\hat{g}}}}} \ar[dd]_{L(\hat{h})} \ar@{=}[r] &
              {LY}\ar[dd]^{L(h)}   &=&  &{}\ar@{}[l]|{\xLeftarrow{L(\yyyy\ast\xxxx ) }} &\xLeftarrow{\hskip .5em \llll _ {{}_{hg}} \hskip .5em  }&
              \\
							&&{}\ar@{}[l]|{\xLeftarrow{L(\yyyy )}} && &&&
															\\
								LZ\ar@{=}[r] &LZ\ar@{=}[r]& LZ && LZ\ar@{=}[r] &LZ\ar@{=}[rr]&& LZ }$$								
\normalsize								
\end{enumerate}								
								
\end{defi}

The composition of pseudofunctors is easily defined. Namely, if $(J, \jjjj): \aaa\to\bbb, (L , \llll ): \bbb\to\ccc $ are pseudofunctors, we define the composition by $ L\circ J := (LJ, (\llll\jjjj ) )$, in which $(\llll\jjjj ) _ {{}_{hg}}:= L(\jjjj _ {{}_{hg}})\cdot \llll _ {{}_{J(h)J(g)}} $ and $(\llll\jjjj )  _ {{}_{X}}:= L(\jjjj _ {{}_{X}})\cdot \llll _ {{}_{JX}} $. This composition is associative and it has trivial identities.

Furthermore, recall that a $2$-functor $L:\bbb\to\ccc $ is just a pseudofunctor $(L, \llll ) $ such that its invertible $2$-cells $\llll _ {{}_f}$ (for every morphism $f$) and $\llll _{{}_X} $ (for every object $X$) are identities.

\begin{defi}[Pseudonatural transformation]\label{pseudonaturaltransformation}
If $ L, E:\bbb\to \ccc $ are pseudofunctors, a \textit{pseudonatural transformation} $\alpha : L\longrightarrow E $ is defined by:
\begin{itemize}
\item For each object $X$ of $\bbb $, a $1$-cell $\alpha _{{}_X}: L X\to E X $ of $\ccc $;
\item For each $1$-cell $g:X\to Y $ of $\bbb $, an invertible $2$-cell $\alpha _{{}_g}: E(g) \alpha _{{}_X}\Rightarrow \alpha _{{}_Y} L(g)  $ of $\ccc $;
\end{itemize} 
such that, if $g, \hat{g}: X\to Y,  f:W\to X  $ are $1$-cells of $\aaa $, and $\xxxx : g\Rightarrow \hat{g} $ is a $2$-cell of $\aaa $, the following equations hold:

\begin{enumerate}
\item Associativity:
\tiny
$$\xymatrix{  L W\ar[dddd]_{L (gf)}\ar@{=}[r] &L W     \ar[r]^{\alpha_{{}_W}} 
                    \ar[dd]_{L(f)} 
                       & 
              {E W }  \ar[dd]^{E (f) }      && L W\ar[dddd]_{L(gf)}\ar[r]^{\alpha _{{}_W}} &E W     \ar[rr]^{E(f)} 
                    \ar[dddd]^{E(gf)} 
                       && 
              {E X}  \ar[dddd]^{E (g)}
							\\
							&&{}\ar@{}[l]|{\xLeftarrow{\alpha _{{}_f}}} && &&&
															\\
              &{L X}\ar@{}[l]|{\xLeftarrow{\llll _ {{}_{gf}}}} \ar[dd]_{L(g)} \ar[r]_{\alpha _{{}_X}} &
              {E X}\ar[dd]^{E(g)}   &=&  &{}\ar@{}[l]|{\xLeftarrow{\alpha _{{}_{gf}} }} &\xLeftarrow{\hskip .5em \eeee _ {{}_{gf}} \hskip .5em  }&
              \\
							&&{}\ar@{}[l]|{\xLeftarrow{\alpha _{{}_{g}}}} && &&&
															\\
								L Y\ar@{=}[r] &L Y\ar[r]_{\alpha _{{}_Y}}& E Y && L Y\ar[r]_{\alpha _{{}_Y}} &E Y\ar@{=}[rr]&& E Y }$$ 							
\normalsize	

\item Identity: 
\tiny
$$\xymatrix{  L W\ar@/_7ex/[dd]|-{L (\id _ {{}_W}) }
                    \ar@{}[dd]|{\xLeftarrow{\llll _ {{}_W}} }
                    \ar@/^7ex/[dd]|-{\id _ {{}_{L W }} }
										\ar[rr]^{\alpha _{{}_W}} 
										&& E W \ar[dd]^{\id _ {{}_ {E W}}}    
&&
L W\ar[rr]^{ \alpha _ {{}_{W}} }\ar[dd]_{L (\id _ {{}_W})}  
&&E W\ar@/_7ex/[dd]|-{E (\id _ {{}_W}) }
                    \ar@{}[dd]|{\xLeftarrow{\eeee _ {{}_W}} }
                    \ar@/^7ex/[dd]|-{\id _ {{}_{E W}} }
\\
&\ar@{}[r]|{\hskip .3em =  \hskip .4em  }   &
 &=& 
&\ar@{}[l]|{\xLeftarrow{ \alpha _{{}_{\id _ {{}_W}}}  } }  & 
\\
 L W\ar[rr]_ {\alpha _ {{}_W}} && E W
 &&
L W\ar[rr]_ {\alpha _{{}_W}} &&E W	 }$$
\normalsize
\item Naturality:						
\tiny
$$\xymatrix{  L X\ar@/_6ex/[dd]_{L (\hat{g}) }
                    \ar@{}[dd]|{\xLeftarrow{L (\xxxx )} }
                    \ar@/^6ex/[dd]^{L (g) }
										\ar[rr]^{\alpha _{{}_X}} 
										&& E X\ar[dd]^{E (g)}    
&&
L X\ar[rr]^{ \alpha _ {{}_{X}} }\ar[dd]_{L (\hat{g})}  
&&E X\ar@/_6ex/[dd]_{E (\hat{g}) }
                    \ar@{}[dd]|{\xLeftarrow{E (\xxxx )} }
                    \ar@/^6ex/[dd]^{E (g) }
\\
&\ar@{}[r]|{\xLeftarrow{\hskip .2em \alpha _{{}_{g}}  \hskip .2em } }   &
 &=& 
&\ar@{}[l]|{\xLeftarrow{\hskip .2em \alpha _{{}_{\hat{g}}}  \hskip .2em } }  & 
\\
 L Y\ar[rr]_ {\alpha _ {{}_Y}} && E Y
 &&
L Y\ar[rr]_ {\alpha _{{}_Y}} &&E Y	 }$$ 
\normalsize								
\end{enumerate}								
								
\end{defi}

Firstly, we define the vertical composition, denoted by $\beta\alpha $, of two pseudonatural transformations $\alpha : L\longrightarrow E, \beta : E\longrightarrow U $ by
$$(\beta\alpha) _ {{}_W} :=\beta _{{}_W}\alpha _ {{}_W} $$
\tiny
$$\xymatrix{  
LW\ar[r]^{\beta _{{}_W}\alpha _ {{}_W}}\ar[d]_{L(f) }\ar@{}[dr]|{\xLeftarrow{(\beta\alpha ) _{{}_f}} } 
&U W\ar@{}[drr]|{:=}
\ar[d]^{U(f)}
&& 
L W\ar[rr]^{\alpha _{{}_W}}\ar[d]^{L (f)}
\ar@{}[drr]|{\xLeftarrow{\alpha  _{{}_f}} } 
&& E W\ar[d]_{E(f) } \ar[r]^{\beta _ {{}_W}} 
\ar@{}[dr]|{\xLeftarrow{\beta  _{{}_f}} }
& U W \ar[d]^{U (f)} 
\\
L X\ar[r]_{\beta _ {{}_X}\alpha _ {{}_X}} & U X
&& L X\ar[rr]_ {\alpha _ {{}_X}} && E X\ar[r]_{\beta _ {{}_X}} &U X
}$$
\normalsize

Secondly, assume that $L, E : \bbb\to \ccc$ and $ G, J : \aaa\to\bbb $ are pseudofunctors. We define the horizontal composition of two pseudonatural transformations  $\alpha : L\longrightarrow E, \lambda : G\longrightarrow J $ by
$\left( \alpha\ast\lambda\right) := (\alpha J)(L\lambda )$, 
in which $\alpha J $ is trivially defined and $(L\lambda ) $ is defined below
\begin{eqnarray*}
\begin{aligned}
(L\lambda ) _{{}_W} &:=& L(\lambda _{{}_W})
\end{aligned}
\qquad\qquad\qquad
\begin{aligned}
(L\lambda )_ {{}_f} &:=& \left(\llll _{{}_{\lambda _{{}_X}G (f)}}\right) ^{-1} \cdot L(\lambda _{{}_f})\cdot \llll _{{}_{J (f)\lambda _{{}_W}}}
\end{aligned}
\end{eqnarray*}

Also, recall that a $2$-natural transformation is just a pseudonatural transformation $\alpha :L\longrightarrow E $ such that its components $\alpha _{{}_g}: E(g) \alpha _{{}_X}\Rightarrow \alpha _{{}_Y} L(g)  $
are identities (for all morphisms $g$).

\begin{defi}[Modification]\label{Modfication}
Let $L, E:\bbb\to \ccc $ be pseudofunctors. If $\alpha , \beta : L\longrightarrow E $ are pseudonatural transformations, a \textit{modification} 
$ \Gamma : \alpha\Longrightarrow \beta $
is defined by the following data:
\begin{itemize}
\item For each object $X$ of $\bbb $, a $2$-cell $\Gamma _{{}_X}: \alpha _ {{}_X}\Rightarrow\beta _{{}_X} $ of $\ccc $;
\end{itemize} 
such that: if $f: W\to X $ is a $1$-cell of $\bbb $, the equation below holds.				
\tiny
$$\xymatrix{  LW\ar@/_6ex/[dd]_{\alpha_{{}_W} }
                    \ar@{}[dd]|{\xRightarrow{\Gamma _{{}_W}} }
                    \ar@/^6ex/[dd]^{\beta _{{}_W} }
										\ar[rr]^{L (f) } && L X\ar[dd]^{\beta _ {{}_X} }    
&&
LW\ar[rr]^{ L (f) }\ar[dd]_{\alpha _ {{}_W}}  &&LX\ar@/_6ex/[dd]_{\alpha _ {{}_X} }
                    \ar@{}[dd]|{\xRightarrow{\Gamma _ {{}_X}} }
                    \ar@/^6ex/[dd]^{\beta _ {{}_X} }
\\
&\ar@{}[r]|{\xRightarrow{\hskip .2em \beta _{{}_{f}}  \hskip .2em } }   &
 &=& 
&\ar@{}[l]|{\xRightarrow{\hskip .2em \alpha _{{}_{f}}  \hskip .2em } }  & 
\\
 E W\ar[rr]_ {E (f)} && E X
 &&
E W\ar[rr]_ {E (f) } &&E X	 }$$ 	
\normalsize

\end{defi}

The three types of compositions of modifications are defined in the obvious way.
Thereby, it is straightforward to verify that, indeed, $2\textrm{-}\CAT $ is a tricategory, lacking strictness/$2$-functoriality of the whiskering. In particular, we denote by $[\aaa , \bbb ]_{PS} $ the $2$-category of pseudofunctors $\aaa\to\bbb $, pseudonatural transformations and modifications.

The bicategorical Yoneda Lemma~\cite{RS80} says that  there is a pseudonatural equivalence $$[\sss , \CAT ]_ {PS} (\sss (a, - ), \mathcal{D})\simeq \DDDD a $$ given by the evaluation at the identity. 

\begin{lem}[Yoneda Embedding~\cite{RS80}]\label{Yoneda}
The Yoneda $2$-functor $\YYYY:\aaa\to [\aaa ^{\op }, \CAT ]_ {PS}$  is locally an equivalence (\textit{i.e.} it induces equivalences between the hom-categories). 
\end{lem}

Considering pseudofunctors $L: \bbb\to \ccc $ and $U:\ccc\to\bbb $, we say that $U$ is right biadjoint to $L$ if we have a pseudonatural equivalence $\ccc (L - , -)\simeq \bbb (- , U-) $. This concept can be also defined in terms of unit and counit as it is done at Definition \ref{ad}. 

\begin{defi}\label{ad}
Let 
$U:\ccc\to\bbb , L:\bbb\to\ccc $
be pseudofunctors. $L$ is
\textit{left biadjoint} to $U$ if there exist
\begin{enumerate}
\item
pseudonatural transformations $\eta :\Id _ {\bbb } \longrightarrow UL$ and
$\varepsilon :LU \longrightarrow \Id _ { \ccc }$
\item
invertible modifications
$s : \id _{L} \Longrightarrow (\varepsilon L)  (L\eta)$
and
$t : (U\varepsilon)  (\eta U) \Longrightarrow \id _{U}$
\end{enumerate}
such that the following $2$-cells are identities~\cite{JGRAY74}:
\small
$$\xymatrix{  \Id _ {\bbb }     \ar[rr]^{\eta} 
                    \ar[dd]_{\eta} 
                       && 
              {UL}  \ar[dd]_{\eta UL}
                    \ar@{=}@/^6ex/[dddr]
                    \ar@{}[dddr]|{\xRightarrow{  tL }} &&
										{LU}  \ar@{=}@/^6ex/[drrr]
                    \ar@{}[drrr]|{\xRightarrow{  \llll _{{}_{U}}^{-1} (Lt)          \llll _ {{}_{(U\varepsilon ) (\eta U)}}  }} 
                    \ar[dr]|{L\eta U}
                    \ar@{}[dddr]|{\xRightarrow{sU}}
                    \ar@{=}@/_6ex/[dddr]   &&& 
                                       \\ 
          & \xRightarrow{ \hskip .5em \eta_{(\eta)}}  &&&   & {LULU}\ar[rr]^{LU\varepsilon}                         
                    \ar[dd]^{\varepsilon LU} &&
              {LU}  \ar[dd]^{\varepsilon}                            
                                        \\
              {UL}  \ar[rr]_{UL\eta }
                    \ar@{=}@/_6ex/[drrr]
                    \ar@{}[drrr]|{\xRightarrow{\uuuu ^{-1}_ {{}_{(L\eta) (\varepsilon L)}} (Us) \uuuu _{{}_{L}}          }} &&
              {ULUL}\ar[dr]|{U\varepsilon L} && &  &\xRightarrow{\hskip .2em \varepsilon_{(\varepsilon)}} &
							\\                                            
          &&& {UL} & & {LU}  \ar[rr]_{\varepsilon}   && \Id _ {\ccc }
					}$$ 							
\normalsize
\end{defi}

\begin{rem}\label{addd}
By definition, if a pseudofunctor $L$ is left biadjoint to $U$, there is at least one associated data $(L\dashv U, \eta , \varepsilon , s, t) $ as described above. Such associated data is called a \textit{biadjunction}. 

Also, every biadjunction $(L\dashv U, \eta , \varepsilon , s, t) $ has an associated pseudonatural equivalence 
$\chi : \ccc(L-, -)\simeq \bbb (-,U-)$,
in which  
\small
\begin{eqnarray*}
\chi _ {{}_{(X,Z)}}: &\ccc(LX, Z)&\to \bbb (X,UZ)\\
                &f        &\mapsto U(f)\eta _{{}_X}\\
								&\mathfrak{m} &\mapsto U(\mathfrak{m})\ast\id _ {{}_{\eta _{{}_X}}}
\end{eqnarray*}
\begin{eqnarray*}
\left(\chi _ {{}_{(g,h)}}\right) _ {{}_f} &:=& \left(\uuuu _{{}_{(hf)Lg}}\ast\id _ {{}_{\eta _ {{}_{X}} }}\right) \cdot \left(\uuuu _ {{}_{hf}}\ast \eta _ {{}_{g}}^{-1}\right) 
\end{eqnarray*}		
\normalsize		
Reciprocally, such a pseudonatural equivalence induces a biadjunction $(L \dashv U, \eta , \varepsilon , s, t)$.
\end{rem}

\begin{rem}\label{unitbiadjunction}
Similarly to the $1$-dimensional case, if $(L\dashv U, \eta , \varepsilon , s, t) $ is a biadjunction, the counit $\varepsilon : LU\longrightarrow \id _ {\ccc }$ is a pseudonatural equivalence if and only if, for every pair $(X,Y) $ of objects  of $\ccc $,
$U_{{}_{X,Y}}: \ccc (X,Y)\to \bbb (UX, UY) $
is an equivalence (that is to say, $U$ is \textit{locally an equivalence}).

The proof is also analogous to the $1$-dimensional case. Indeed, given a pair $(X,Y) $ of objects in $\bbb $, the composition of functors 
$$\xymatrix{ \bbb (X, Y)\ar[r]^-{\bbb (\varepsilon _ {{}_{X}}, Y) } & \bbb (LUX, Y)\ar[r]^-{\chi _{{}_{(UX, Y) }} }  & \bbb (LX, LY) 
}$$ 
is obviously isomorphic to $U_{{}_{X,Y}}:\ccc (X,Y)\to \bbb (UX, UY)$. Since $\chi _{{}_{(UX, Y) }}$ is an equivalence, $\varepsilon _ {{}_{X}}$ is an equivalence for every object $X$ (that is to say, it is a pseudonatural equivalence) if and only if $U $ is locally an equivalence. Dually, the unit of this biadjunction is a pseudonatural equivalence if and only if $L$ is locally an equivalence.
\end{rem}

\begin{rem}
Recall that, if the modifications $s,t$ of a biadjunction $(L\dashv  U, \eta , \varepsilon , s, t) $ are identities, $L, U $ are $2$-functors and $\eta , \varepsilon $ are $2$-natural transformations, then $L$ is left $2$-adjoint to $U$ and 
$(L\dashv U, \eta , \varepsilon ) $ is a $2$-adjunction.
\end{rem}

If it exists, a birepresentation of a pseudofunctor $\mathcal{U}:\ccc\to\CAT $ is an object $X$ of $\ccc$ endowed with a pseudonatural equivalence $\ccc (X, -)\simeq \mathcal{U} $.  When $\mathcal{U}$ has a birepresentation, we say that $\mathcal{U}$ is birepresentable. Moreover, in this case, by Lemma \ref{Yoneda}, its birepresentation is unique up to equivalence. 

\begin{lem}[\cite{RS80}]\label{rep}
Assume that $\mathcal{U}: \ccc \to [\bbb ^{\op } , \CAT]_ {PS} $ is a pseudofunctor such that, for each object $X$ of $\ccc $,  $\mathcal{U}X$ has a birepresentation $e_X: \mathcal{U}X\simeq \bbb (- , UX ) $. Then there is a pseudofunctor $ U: \ccc \to\bbb $ such that the pseudonatural equivalences $e_X $ are the components of a pseudonatural equivalence $\mathcal{U}\simeq \bbb (-, U-) $, in which $\bbb (-,U-) $ denotes the pseudofunctor
$$
\ccc  \to  [\bbb ^{\op }, \CAT]_ {PS}:\mbox{  } X \mapsto  \bbb (-, UX)  
$$
\end{lem}
As a consequence, a pseudofunctor $L: \bbb \to \ccc $ has a right biadjoint if and only if, for each object $X$ of $\ccc $, the pseudofunctor $\ccc ( L-, X) $ is birepresentable. \textit{Id est}, for each object $X$, there is an object $UX$ of $\bbb $ endowed with a pseudonatural equivalence $\ccc ( L-, X) \simeq \bbb (-,UX)$.

The natural notion of limit in our context is that of (weighted) bilimit~\cite{RS80, RS87}. Namely, assuming  that $\sss $ is a small $2$-category, if $ \WWWW : \sss \to \CAT, \DDDD :\sss\to\aaa $ are pseudofunctors, the (weighted) bilimit, denoted herein by $\left\{ \WWWW, \DDDD\right\} _ {\bi} $, when it exists, is a birepresentation of the $2$-functor
$$ \aaa ^{\op }\to  \CAT :\mbox{ } X \mapsto   [\sss , \CAT ]_{PS}(\WWWW, \aaa (X, \DDDD -) ).$$
Since, by the (bicategorical) Yoneda Lemma,  $\left\{ \WWWW, \DDDD\right\}_ {\bi} $ is unique up to equivalence, we sometimes refer to it as \textit{the} (weighted) bilimit.

Finally, if $\WWWW$ and $\DDDD $ are $2$-functors, recall that the (strict) weighted limit $\left\{ \WWWW, \DDDD\right\} $ is, when it exists, a $2$-representation of the $2$-functor $X \mapsto   [\sss , \CAT ](\WWWW , \aaa (X, \DDDD -) )$, in which $[\sss , \CAT ]$ is the $2$-category of $2$-functors $\sss\to\CAT $, $2$-natural transformations and modifications~\cite{RS76}.

It is easy to see that $\CAT $ is bicategorically complete. More precisely, if   $\WWWW : \sss \to \CAT $ and $ \DDDD :\sss\to \CAT $ are pseudofunctors, then $$\left\{ \WWWW, \DDDD\right\}_ {\bi}  \simeq  [\sss , \CAT ]_{PS}(\WWWW, \DDDD ) .$$ Moreover, from the bicategorical Yoneda Lemma of \cite{RS80}, we get the  (strong) bicategorical Yoneda Lemma. 

\begin{lem}[(Strong) Yoneda Lemma]\label{Yoneda2}
Let $\DDDD : \sss\to\aaa $ be a pseudofunctor between $2$-categories. There is a pseudonatural equivalence $\left\{ \sss (a, -), \DDDD\right\} _ {\bi}\simeq \DDDD a $. 
\end{lem}
\begin{proof}
By the bicategorical Yoneda Lemma, we have a pseudonatural equivalence (in $X$ and $a$)  
$$[\sss , \CAT ] _ {PS} (\sss (a, -), \aaa (X, \DDDD -))\simeq \aaa (X, \DDDD a) .$$
Therefore $\DDDD a$ is \textit{the} bilimit $\left\{ \sss (a, -), \DDDD\right\} _{\bi } $. 
\end{proof} 

Recall that the usual (enriched) Yoneda embedding $\aaa\to\left[ \aaa ^{\op }, \CAT\right] $ preserves and reflects weighted limits. In the $2$-dimensional case, we get a similar result. 

\begin{lem}\label{Yonedaargument}
The Yoneda embedding $\YYYY : \aaa\to\left[ \aaa ^{\op }, \CAT\right] _ {PS} $ preserves and reflects weighted bilimits.
\end{lem}
\begin{proof}
By definition, a weighted bilimit $\left\{ \WWWW, \DDDD \right\} _ {\bi } $ exists if and only if, for each object $X$ of $\aaa $, 
$$\aaa (X, \left\{ \WWWW, \DDDD \right\} _ {\bi })\simeq  \left[ \aaa, \CAT\right] _ {PS}(\WWWW, \aaa (X, \DDDD - ))\simeq \left\{ \WWWW, \aaa (X, \DDDD - ) \right\} _ {\bi }.$$
By the pointwise construction of weighted bilimits, this means that $\left\{ \WWWW, \DDDD \right\} _ {\bi } $ exists if and only if
$\YYYY\left\{ \WWWW, \DDDD \right\} _{\bi }\simeq \left\{ \WWWW, \YYYY\circ\DDDD \right\} _{\bi }$.
This proves that $\YYYY $ reflects and preserves weighted bilimits.
\end{proof}

\begin{rem}
Let $\sss $ be a small $2$-category and $\DDDD : \sss\to\aaa $ be a pseudofunctor. Consider the pseudofunctor   
$$
[\sss, \ccc ]_ {PS} \to  [ \aaa  ^{\op },  \CAT ]_{PS} :\mbox{ }\WWWW  \mapsto   \mathbbmss{D}_\WWWW  $$ 
in which the $2$-functor $\mathbbmss{D}_\WWWW$ is given by  $X \mapsto   [\sss , \CAT ]_{PS}(\WWWW, \aaa (X, \mathcal{D}-) ) $. 
By Lemma \ref{rep}, we conclude that it is possible to get a pseudofunctor $\left\{ -, \mathcal{D} \right\} _ {\bi} $ defined in a full sub-$2$-category of $[\sss , \CAT ]_{PS}$ of weights $\WWWW:\sss\to\CAT $ such that $\aaa$ has the bilimit $\left\{ \WWWW,\mathcal{D}\right\} _ {\bi }$. 
\end{rem}

\section{Descent Objects}\label{Descent}

In this section, we describe the $2$-categorical limits called \textit{descent objects}. We need both constructions, strict descent objects and descent objects \cite{RS87}. Our domain $2$-category, denoted by $\Delta $, is the dual of that defined at Definition 2.1 in \cite{CME}.

\begin{defi}\label{Presentation}
We denote by $\dot{\Delta } $ the $2$-category generated by the diagram
\[\xymatrix{  \mathsf{0} \ar[rr]^-d && \mathsf{1}\ar@<1.7 ex>[rrr]^-{d^0 }\ar@<-1.7ex>[rrr]_-{d^1 } &&& \mathsf{2}\ar[lll]|-{s^0}
\ar@<1.7 ex>[rrr]^{\partial ^0 }\ar[rrr]|-{\partial ^1}\ar@<-1.7ex>[rrr]_{\partial ^2} &&& \mathsf{3} }\] 
with the invertible $2$-cells:
\small
\begin{eqnarray*}
\sigma_{ik} &:&  \partial ^k d ^i\cong \partial^{i}d^{k-1} ,
                    \hspace{1mm}\mbox{if  }\hspace{1mm} i<k \\
 n_0        &:&  s^0d^0\cong \id _{{}_\mathsf{1}}  \\
 n_1        &:&   \id _{{}_\mathsf{1}}\cong s^{0}d^{1}\\
 \vartheta        &:&  d^1d\cong d^0d
\end{eqnarray*}
\normalsize
satisfying the equations below: 
\begin{itemize}
\item Associativity:
\small
$$\xymatrix{  \mathsf{0}\ar[r]^-{d}\ar[d]_-{d}\ar@{}[rd]|-{\xRightarrow{\hskip .1em \vartheta \hskip .1em } }&
\mathsf{1}\ar[d]|-{d^0}\ar[r]^-{d^0}\ar@{}[rd]|-{\xRightarrow{\hskip .1em \sigma_{01} \hskip .1em }}&
\mathsf{2}\ar[d]^-{\partial ^0}\ar@{}[rrdd]|-{=}&
&
\mathsf{3}\ar@{}[rd]|{\xRightarrow{\hskip .1em \sigma_{02} \hskip .1em }}&
\mathsf{2}\ar[l]_{\partial ^0}\ar@{}[rdd]|{\xRightarrow{\hskip .1em \vartheta \hskip .1em }}\ar@{=}[r]&
\mathsf{2}
\\
\mathsf{1}\ar[r]|-{d^1}\ar[d]_-{d^1}\ar@{}[rrd]|-{\xRightarrow{\hskip .1em \sigma _ {12} \hskip .1em }}&
\mathsf{2}\ar[r]|-{\partial ^1}&
\mathsf{3}\ar[d]^-{\id _ {{}_\mathsf{3}}}&
&
\mathsf{2}\ar@{}[rd]|-{\xRightarrow{\hskip .1em \vartheta \hskip .1em }}\ar[u]^-{\partial ^2}&
\mathsf{1}\ar[l]|-{d^0}\ar[u]|-{d^1}&
\\
\mathsf{2}\ar[rr]_{\partial ^2}&&
\mathsf{3}&
&
\mathsf{1}\ar[u]^-{d^1}&
\mathsf{0}\ar[l]^{d}\ar[u]|-{d}\ar[r]_ {d}&
\mathsf{1}\ar[uu]_{d^0}   
}$$
\normalsize
\item Identity:
\small
$$\xymatrix{  \mathsf{0}     \ar[rr]^{d} 
                    \ar[dd]_{d} 
                      && 
              {\mathsf{1}}  \ar[dd]_{d^1}
                    \ar@{=}@/^4ex/[dddr]
                    \ar@{}[dddr]|{\xLeftarrow{n_1}} &&&
        \mathsf{0}    \ar@/_3ex/[ddd]_{d}
                    \ar@{}[ddd]|=
                    \ar@/^3ex/[ddd]^{d}
                                        \\ 
          & \xLeftarrow{\hskip .1em \vartheta \hskip .1em }  && &=                            
                                        \\
              {\mathsf{1}}  \ar[rr]_{d^0}
                    \ar@{=}@/_4ex/[drrr]
                    \ar@{}[drrr]|{\xLeftarrow{n_0}} &&
              {\mathsf{2}}\ar[dr]|{s^0}  \\                                            
          &&& {\mathsf{1}} && {\mathsf{1}} }$$ 
\normalsize					
\end{itemize}
The $2$-category $\Delta $ is, herein, the full sub-$2$-category of $\dot{\Delta } $ with objects $\mathsf{1}, \mathsf{2}, \mathsf{3}$. We denote the inclusion by $\j: \Delta\to\dot{\Delta } $.
\end{defi}

\begin{rem}
In fact, the $2$-category $\dot{\Delta } $ is the locally preordered $2$-category freely generated by the diagram and $2$-cells described above. Moreover, $\Delta $ is the $2$-category freely generated by the corresponding diagram and the $2$-cells $\sigma_{01}, \sigma_{02}, \sigma_{12}, n_0, n_1 $.
\end{rem}

Let $\aaa $ be a $2$-category and $\AAAA : \Delta\to \aaa $ be a $2$-functor. If the weighted bilimit $\left\{ \dot{\Delta } (\mathsf{0}, \j -), \AAAA\right\} _ {\bi } $ exists, we say that  $\left\{ \dot{\Delta } (\mathsf{0}, \j -), \AAAA\right\} _ {\bi } $ is the \textit{descent object} of $\AAAA $. Analogously, when it exists, we call the (strict) weighted $2$-limit $\left\{ \dot{\Delta } (\mathsf{0}, \j -), \AAAA\right\} $ the \textit{strict descent object} of $\AAAA $.

Assuming that $\DDDD : \dot{\Delta }\to\aaa $ is a pseudofunctor, we have a pseudonatural transformation $\dot{\Delta } (\mathsf{0}, \j - )\longrightarrow \aaa(\DDDD\mathsf{0}, \DDDD\circ\j - )$ given by the evaluation of $\DDDD $. By the definition of weighted bilimit, if $\DDDD\circ\j $ has a descent object,  this pseudonatural transformation induces a comparison $1$-cell $$\DDDD\mathsf{0}\to\left\{ \dot{\Delta } (\mathsf{0}, \j -), \DDDD\circ\j\right\} _ {\bi }. $$
Analogously, if $\DDDD $ is a $2$-functor, we get a comparison $\DDDD\mathsf{0}\to\left\{ \dot{\Delta } (\mathsf{0}, \j -), \DDDD\circ\j\right\} $, provided that the strict descent object of $\DDDD\circ\j $ exists.

\begin{defi}[Effective Descent Diagrams]
We say that a $2$-functor $\DDDD : \dot{\Delta }\to \aaa $ is of \textit{effective descent} if $\aaa $ has the descent object of $\DDDD\circ\j $ and the comparison $\displaystyle\DDDD\mathsf{0}\to \left\{ \dot{\Delta } (\mathsf{0}, \j -), \DDDD\circ\j\right\} _ {\bi } $
is an equivalence. 

We say that $\DDDD $ is of \textit{strict descent} if $\aaa$ has the strict descent object of $\DDDD\circ\j $ and the comparison
$\DDDD\mathsf{0}\to \left\{ \dot{\Delta } (\mathsf{0}, \j -), \DDDD\circ\j\right\} $
is an isomorphism.
\end{defi}
 
\begin{lem}\label{strictequivalence}
Strict descent objects are descent objects. Thereby, strict descent diagrams are of effective descent as well.

Also, if $\aaa $ has strict descent objects, a $2$-functor $\DDDD : \dot{\Delta }\to \aaa $ is of effective descent if and only if the comparison
$\DDDD\mathsf{0}\to \left\{ \dot{\Delta } (\mathsf{0}, \j -), \DDDD\circ\j\right\} $
is an equivalence.
\end{lem}

\begin{lem}\label{triiviall}
Assume that $\AAAA , \BBBB , \DDDD : \dot{\Delta } \to \aaa $ are $2$-functors. If there are a $2$-natural isomorphism $\AAAA\longrightarrow\BBBB $ and a pseudonatural equivalence $\BBBB\longrightarrow \DDDD $, then
\begin{itemize}
  \item $\AAAA$ is of strict descent if and only if $\BBBB $ is of strict descent;
	\item $\BBBB $ is of effective descent if and only if $\DDDD $ is of effective descent.
\end{itemize}	
\end{lem}

We say that an effective descent diagram $\DDDD :\dot{\Delta }\to\bbb $ is \textit{preserved} by a pseudofunctor $L:\bbb\to\ccc $ if $L\circ \DDDD $ is of effective descent. Also, $\DDDD :\dot{\Delta }\to\bbb $ is said to be  an \textit{absolute effective descent} diagram if $ L\circ\DDDD $ is of effective descent  for any pseudofunctor $L$. 

In this setting, a pseudofunctor $L:\bbb\to\ccc $ is said to \textit{reflect absolute effective descent diagrams} if, whenever a $2$-functor $\DDDD:\dot{\Delta }\to\bbb $ is such that $L\circ \DDDD $ is an absolute effective descent diagram, $\DDDD $ is of effective descent.  
Moreover, we say herein that a pseudofunctor $L:\bbb\to\ccc $ \textit{creates absolute effective descent diagrams} if $L$ reflects absolute effective descent diagrams and, whenever a diagram $\AAAA : \Delta \to \bbb $ is such that $L\circ \AAAA \simeq \DDDD\circ\j $ for some absolute effective descent diagram $\DDDD : \dot{\Delta } \to \ccc $, there is a diagram $\BBBB : \dot{\Delta }\to\bbb $ such that $L\circ\BBBB \simeq \DDDD $ and $\BBBB\circ\j = \AAAA $.

Recall that right $2$-adjoints preserve strict descent diagrams and right biadjoints preserve effective descent diagrams. Also, the usual (enriched) Yoneda embedding $\aaa\to\left[ \aaa ^{\op }, \CAT\right]$ preserves and reflects strict descent diagrams, and, from Lemma \ref{Yonedaargument},  we get:

\begin{lem}\label{Yonchange}
The Yoneda embedding 
$\YYYY : \aaa\to\left[ \aaa ^{\op }, \CAT\right] _ {PS} $
preserves and reflects effective descent diagrams. 
\end{lem}

\begin{rem}\label{codescent}
The dual notion of descent object is that of codescent object, described by Lack~\cite{SLACK2002} and Le Creurer, Marmolejo, Vitale~\cite{CME}. It is, of course,  the descent object in the opposite $2$-category.
\end{rem}

\begin{rem}\label{strictdescentstrict}
The $2$-category $\CAT $ is $\CAT$-complete. In particular, $\CAT $ has strict descent objects. More precisely, if 
$\AAAA : \Delta\to\CAT $
is a $2$-functor, then
$$\left\{ \dot{\Delta }(\mathsf{0}, \j -), \AAAA \right\} \cong \left[ \Delta , \CAT\right] \left( \dot{\Delta }(\mathsf{0}, \j -), \AAAA \right) .$$
Thereby, we can describe  the strict descent object of $\AAAA :\Delta\to\CAT $ explicitly as follows:

\begin{enumerate}
\item Objects are $2$-natural transformations $\mathsf{f}: \dot{\Delta }(\mathsf{0}, \j -)\longrightarrow \AAAA $.
 We have a bijective correspondence between such $2$-natural transformations and pairs $(f, \varrho _ {{}_{\mathsf{f}}})$ in which $f$ is an object of $ \AAAA\mathsf{1} $ and $\varrho _ {{}_{\mathsf{f}}}: \AAAA (d^1)f\to\AAAA (d^0)f $ is an isomorphism in $ \AAAA\mathsf{2} $ satisfying the following equations:	
\begin{itemize}
\item Associativity:
\small
$$\left(\AAAA (\partial ^0)(\varrho _ {{}_{\mathsf{f} }} )\right) \left( \AAAA (\sigma _ {{}_{02}}) _ {{}_{f}}\right)\left(\AAAA (\partial ^2)(\varrho _ {{}_{\mathsf{f}}} )\right)\left(\AAAA (\sigma _ {{}_{12}} ) ^{-1}_ {{}_f}\right) = \left(\AAAA (\sigma _ {{}_{01}}) _ {{}_{f}}\right)\left(\AAAA(\partial ^1)(\varrho _ {{}_{\mathsf{f}}})\right)   $$
\normalsize
\item Identity:
\small
$$\left(\AAAA(n_0) _ {{}_f}\right)\left(\AAAA(s^0) (\varrho _ {{}_{\mathsf{f}}}) \right)\left(\AAAA(n_1) _ {{}_f}\right) = \id _ {{}_{f}} $$
\normalsize
\end{itemize}
If $\mathsf{f}: \dot{\Delta }(\mathsf{0}, \j -)\longrightarrow \AAAA $
 is a $2$-natural transformation, we get such pair by the correspondence 
$\mathsf{f}\mapsto (\mathsf{f} _ {{}_{\mathsf{1} }}(d), \mathsf{f} _ {{}_{\mathsf{2} }}(\vartheta )) $.

\item The morphisms are modifications. In other words, a morphism $\mathsf{m} : \mathsf{f}\to\mathsf{h} $ is determined by a morphism $\mathfrak{m}: f\to h $ such that $\AAAA (d^0)(\mathfrak{m} )\varrho _ {{}_\mathsf{f}} =  \varrho _ {{}_h}\AAAA (d^1)(\mathfrak{m} ) $.
\end{enumerate}
\end{rem}

\section{Biadjoint Triangles}\label{Main}
In this section, we give our main theorem on biadjoint triangles, Theorem \ref{MAIN}, and its strict version, Theorem \ref{MAINstrict}. Let $L: \bbb\to \ccc $ and $U: \ccc\to\bbb $ be pseudofunctors, and
$(L\dashv U, \eta , \varepsilon , s, t)$ be a biadjunction. We denote by $\chi : \ccc(L-, -)\simeq \bbb (-,U-)$ its associated pseudonatural equivalence as described in Remark \ref{addd}. 

\begin{defi}\label{maindiagram}							
In this setting, for every pair $(X,Y)$ of objects  of $\bbb $, we have an induced diagram $\DDDD_Y^X :\dot{\Delta }\to \CAT $ 
\small
\begin{equation*}\tag{$\DDDD_Y^X$}
\xymatrixcolsep{1.7pc}\xymatrix{ \bbb (X, Y) \ar[d]|-{L_{{}_{X,Y}}}\\  \ccc (LX, LY)\ar@<-2.5ex>[rrr]_-{\ccc(LX, L(\eta _{{}_Y} ))}\ar@<2.5ex>[rrr]^-{L_{{}_{{}_{X,ULY}}}\circ\hspace{0.2em} \chi_ {{}_{{}_{(X,LY)}}}} &&& \ccc (LX, LULY )\ar[lll]|-{\ccc (LX,\varepsilon _ {{}_{LY}})}
\ar@<-2.5 ex>[rrrr]_-{\ccc (LX, LUL(\eta _{{}_Y} ))}\ar[rrrr]|-{\ccc (LX, L(\eta _ {{}_{ULY}}))}\ar@<2.5ex>[rrrr]^-{L_{{}_{{}_{X,(UL)^2Y}}}\circ \hspace{0.2em}\chi_ {{}_{{}_{(X,LULY)}}}} &&&& \ccc (LX, L(UL)^2Y) }
\end{equation*}
\normalsize
in which the images of the $2$-cells of $\dot{\Delta } $ by $\DDDD_Y^X: \dot{\Delta }\to\CAT $ are defined as:   
\begin{eqnarray*}
\begin{aligned}
&\DDDD_Y^X(\vartheta ) _{{}_g}      & : = &  L\left(\eta _ {{}_g}^{-1}\right)\cdot \llll _{{}_{\eta _ {{}_Y} g}} \\
&\DDDD _Y^X (\sigma _{12}) _{{}_f}     &    : =& \left(L\eta\right) _ {{}_{\eta _ {{}_Y} }}\ast \id _ {{}_f}\\
&\DDDD _Y ^X(n_1)_{{}_f} &: = &s_{{}_Y}\ast \id _ {{}_f} 
\end{aligned}
\qquad
\begin{aligned} 
&\DDDD_Y^X (\sigma _{01})_{{}_f}              &  : = &  \llll _ {{}_{UL(U(f)\eta _ {{}_{X}} )\eta _ {{}_{X}}}}\cdot \left( L\eta \right) _ {{}_{U(f)\eta _ {{}_{X}} }}^{-1} 
 \\  
&\DDDD_Y^X(\sigma _{02}) _{{}_f}               &: = &   L\left(\uuuu _ {{}_{L(\eta _ {{}_Y}) f }}\ast \id _ {{}_{\eta _ {{}_X}}}\right)\cdot \llll _ {{}_{UL(\eta _ {{}_Y})L(U(f)\eta _ {{}_X})}}\\
&\DDDD_Y^X(n_0)_{{}_f} &: = &\left( \id _ {{}_f}\ast s_ {{}_X} ^{-1}  \right)\cdot \left( \varepsilon _ {{}_f} ^{-1}\ast \id _ {{}_{\eta _ {{}_X}}}\right)\cdot \left( \id _ {{}_{\varepsilon _ {{}_{LY}}}}\ast \llll _ {{}_ {U(f)\eta _ {{}_X}}}^{-1}\right)   
\end{aligned}
\end{eqnarray*}  
\end{defi}
We claim that \textit{$\DDDD^X_Y $ is well defined.} In fact, by the axioms of naturality and associativity of Definition~\ref{pseudonaturaltransformation} (of pseudonatural transformation), for every morphism $g\in \bbb(X,Y) $, we have the equality
\small
$$\xymatrix{  LX\ar[d]_{L(\eta _ {{}_{X}})}\ar[r]^{L(g)}\ar@{}[rd]|{\xLeftarrow{\gamma }} &
LY \ar[d]^{L(\eta _ {{}_{Y}})}\ar[rd]^{L(\eta _ {{}_{Y}})} &
&
LX\ar[d]_{L(\eta _ {{}_{X}}) }\ar[r]^{L(g)}\ar[rd]|{L(\eta _ {{}_{X}} )}\ar@{}[rrd]|{\xLeftarrow{\gamma } }\ar@{}[rdd]|{\xLeftarrow{\left( L\eta  \right)_ {{}_{\eta _ {{}_X}}}^{-1}}}&
LY\ar[rd]^{L(\eta _ {{}_Y})}&  
\\
LULX\ar[r]_{LUL(g)}\ar[rd]_{LUL(\eta _ {{}_{X}} )}\ar@{}[rrd]|{\xLeftarrow{\widehat{LU(\gamma )} }}&
LULY\ar[rd]|{LUL(\eta _ {{}_{Y}}) }\ar@{}[r]|{\xLeftarrow{\left( L\eta  \right)_ {{}_{\eta _ {{}_Y}}}^{-1} }}&
LULY\ar[d]^{L(\eta _{{}_{ULY}})}\ar@{}[r]|{=}&
LULX\ar[rd]_{LUL(\eta _ {{}_X}) } &
LULX\ar[d]|{L(\eta _ {{} _ {ULX}})}\ar[r]^{LUL(g)}\ar@{}[rd]|{\xLeftarrow{ (L\eta )_ {{}_ {UL(g)}}^{-1} } }&
LULY\ar[d] ^{L(\eta _ {{}_ {ULY}})} 
\\
&
LULULX\ar[r]_{LULUL(g)}&
LULULY &
&
LULULX \ar[r]_{LULUL(g)}&
LULULY  }$$ 
\normalsize
in which 
\small
\begin{eqnarray*}
\begin{aligned}
\gamma  &: = \llll _{{}_{UL(g)\eta _ {{}_X}}}^{-1}\cdot\DDDD_Y^X(\vartheta ) _{{}_g} = (L\eta ) _ {{}_g} ^{-1} 
\end{aligned}
\qquad\qquad
\begin{aligned}
\widehat{LU(\gamma )} &: = (\llll\uuuu) _ {{}_{LUL(g)L(\eta _{{}_{X}})}}^{-1}\cdot LU(\gamma )\cdot (\llll\uuuu) _ {{}_{L(\eta _{{}_{X}})L(g)}}\\
\end{aligned}
\end{eqnarray*}
\normalsize
By the definition of $\DDDD_Y^X $ given above, this is the same as saying that the equation 
\tiny
$$\xymatrix{ 
\DDDD_Y^X \mathsf{3}\ar@{=}[rr]\ar@{}[rrdd]|{\xRightarrow{\DDDD_Y^X(\sigma _ {12})^{-1}}}&&
\DDDD_Y^X \mathsf{3}\ar@{}[rd]|{\xRightarrow{\DDDD_Y^X(\sigma _ {02})}}&
\DDDD_Y^X \mathsf{2}\ar@{=}[r]\ar@{}[rdd]|{\xRightarrow{\DDDD_Y^X(\vartheta )}}\ar[l]_{\DDDD_Y^X(\partial ^0)}&
\DDDD_Y^X \mathsf{2}\ar@{}[rdd]|{=}
&
\DDDD_Y^X \mathsf{0}\ar[r]^{\DDDD ^X_Y (d) }\ar[dd]|{\DDDD _Y^X (d) }\ar@{}[rdd]|{\xRightarrow{\DDDD^X_Y(\vartheta) }}&
\DDDD_Y^X \mathsf{1}\ar[dd]|{\DDDD_Y^X(d^0)}\ar[rr]^{\DDDD_Y^X(d^0)}\ar@{}[rrdd]|{\xRightarrow{\DDDD_Y^X(\sigma _ {01})}}&&
\DDDD_Y^X \mathsf{2}\ar[dd]|{\DDDD_Y^X(\partial ^0 )}
\\
&&
\DDDD_Y^X \mathsf{2}\ar@{}[rd]|{\xRightarrow{\DDDD_Y^X(\vartheta)}}\ar[u]^{\DDDD_Y^X(\partial ^2)}&
\DDDD_Y^X \mathsf{1}\ar[l]^{\DDDD_Y^X(d^0)}\ar[u]_{\DDDD_Y^X(d^1)}&
&&&&
\\
\DDDD_Y^X \mathsf{2}\ar[uu]|{\DDDD_Y^X(\partial ^1) }&&
\DDDD_Y^X \mathsf{1} \ar[ll]^{\DDDD_Y^X(d^1) }\ar[u]^{\DDDD_Y^X(d^1)}&
\DDDD_Y^X \mathsf{0} \ar[l]^{\DDDD_Y^X(d)}\ar[u]_{\DDDD_Y^X(d)}\ar[r]_{\DDDD_Y^X(d)}&
\DDDD_Y^X \mathsf{1}\ar[uu]|{\DDDD_Y^X (d^0) } 
&
\DDDD_Y^X \mathsf{1} \ar[r]_{\DDDD_Y^X(d^1)}&
\DDDD_Y^X \mathsf{2} \ar[rr]_{\DDDD_Y^X(\partial ^1) }&&
\DDDD_Y^X \mathsf{3} 
 }$$ 
\normalsize
holds, which is equivalent to the usual equation of associativity given in Definition \ref{Presentation}. Also, by the naturality of the modification $s : \id _{L} \Longrightarrow (\varepsilon L)  (L\eta)$ (see Definition~\ref{Modfication}), for every morphism $g\in\bbb (X,Y)$, the pasting of $2$-cells 
\small
$$\xymatrix{ 
LX\ar[rrrr]^{L(g)}\ar[rd]|{L(\eta _ {{}_X})}\ar@{=}[dd]
&&&&
LY\ar[ld]|{L(\eta _ {{}_Y}) }\ar@{=}[dd]\\
&
LULX\ar@{}[l]|{\xLeftarrow{s_{{}_X} ^{-1}}}\ar[rr]|{LUL(g)}\ar[ld]|{\varepsilon _ {{}_{LX}} }
&
\ar@{}[u]|{\xLeftarrow{(L\eta)_{{}_g}^{-1}}}\ar@{}[d]|{\xLeftarrow{(\varepsilon L) _ {{}_g}^{-1}}}
&
LULY\ar@{}[r]|{\xLeftarrow{s_{{}_Y}} }\ar[rd]|{\varepsilon _ {{}_{LY}}}
&
\\
LX\ar[rrrr]_{L(g)}
&&&&
LY  
}$$
\normalsize
is equal to the identity $L(g)\Rightarrow L(g) $ in $\ccc $. This is equivalent to say that 
\tiny
$$\xymatrix{  \DDDD_Y^X\mathsf{0}     \ar[rr]^-{\DDDD_Y^X(d)} 
                    \ar[dd]|-{\DDDD_Y^X(d)} 
                      && 
              {\DDDD_Y^X\mathsf{1}}  \ar[dd]|-{\DDDD_Y^X(d^1)}
                    \ar@{=}@/^7ex/[dddr]
                    \ar@{}[dddr]|-{\xLeftarrow{\DDDD_Y^X(n_1)}} &&&
       \DDDD_Y^X\mathsf{0}    \ar@/_3ex/[ddd]_-{\DDDD_Y^X(d)}
                    \ar@{}[ddd]|=
                    \ar@/^3ex/[ddd]^-{\DDDD_Y^X(d)}
                                        \\ 
          & \xLeftarrow{\hskip .1em \DDDD_Y^X(\vartheta )}  && &=                            
                                        \\
              {\DDDD_Y^X\mathsf{1}}  \ar[rr]|{\DDDD_Y^X(d^0)}
                    \ar@{=}@/_7ex/[drrr]
                    \ar@{}[drrr]|{\xLeftarrow{\DDDD_Y^X(n_0)}} &&
              {\DDDD_Y^X\mathsf{2}}\ar[dr]|-{\DDDD_Y^X(s^0)}  \\                                            
          &&& {\DDDD_Y^X\mathsf{1}} && {\DDDD_Y^X\mathsf{1}} }$$ 
\normalsize
holds,
which is the usual identity equation of Definition \ref{Presentation}. Thereby it completes the proof that indeed $\DDDD _ Y^X $ is well defined.

As in the enriched case, we also need to consider another special $2$-functor induced by a biadjoint triangle.

\begin{defi}\label{trianglediagram}
Let $(E\dashv R, \rho , \mu , v,w)$ and $(L\dashv U, \eta , \varepsilon , s,t)$ be biadjunctions such that the triangle
$$\xymatrix{  \aaa\ar[rr]^-{J}\ar[dr]_-{E}&&\bbb\ar[dl]^-{L}\\
&\ccc  & }$$
is commutative.
In this setting, for each object $Y$ of $\bbb $, we define the $2$-functor
$\AAAA _ Y : \Delta\to \aaa  $
\small
\begin{equation*}\tag{$\AAAA_Y$}
\xymatrix{  RLY\ar@<-2.5ex>[rrrr]_-{RL(\eta _{{}_Y})}\ar@<2.5ex>[rrrr]^-{RL(U(\mu _{{}_{LY}})\eta _{{}_{JRLY}})\rho _ {{}_{RLY}}} &&&&  RLULY \ar[llll]|-{R(\varepsilon _{{}_{LY}})}
\ar@<-2.5 ex>[rrrrr]_-{RLUL(\eta_{{}_{Y}})}\ar[rrrrr]|-{RL(\eta_{{}_{ULY}})}\ar@<2.5ex>[rrrrr]^-{ RL(U(\mu _{{}_{LULY}})\eta _ {{}_{JRLULY}})\rho _{{}_{RLULY}}} &&&&& RLULULY }
\end{equation*}
\normalsize
in which
\small
\begin{eqnarray*}
\AAAA _Y (\sigma _{12})        &: =& (RL\eta)_{{}_{\eta _ {{}_{Y}} }}\\
\AAAA _Y (n_1) &: =& \rrrr _ {{}_{\varepsilon _ {{}_{LY}}\cdot L(\eta _ {{}_{Y}}) }}^{-1} R(s_ {{}_{Y}})\cdot \rrrr _ {{}_{LY}}\\
\AAAA _Y (n_0)        &: =& 
\left(w_{{}_{LY}}\right)\\
&&\cdot\left(\id _ {{}_{R(\mu _ {{}_{LY}} )}}\ast \left( \rrrr _ {{}_{ERLY}}^{-1}\cdot R(s _ {{}_{JRLY}}^{-1})\cdot \rrrr _ {{}_{\varepsilon _ {{}_{ERLY}}L(\eta _ {{}_{JRLY}}) }}\right)\cdot \id _ {{}_{\rho _ {{}_{RLY}} }}\right)\\
&&\cdot
\left( (R\varepsilon )^{-1}_ {{}_{\mu _ {{}_{LY}} }}\ast \id _ {{}_{RL(\eta _ {{}_{JRLY}})\rho _ {{}_{RLY}} }}\right)\\
&&\cdot 
\left( \id _ {{}_{R(\varepsilon _ {{}_{LY}}) }}\ast (\rrrr\llll ) ^{-1}_ {{}_{U(\mu _ {{}_{LY}})\eta _ {{}_{JRLY}}  }}\ast \id _ {{}_{\rho _ {{}_{RLY}} }}\right)
\end{eqnarray*}
\begin{eqnarray*}
\AAAA _Y (\sigma _{02})        &: =& 
\left( (\rrrr\llll)_{{}_{ U(\mu _ {{}_{RLULY}})\eta _ {{}_{JRLULY}} }}\ast \id _ {{}_{\rho _ {{}_{RLULY}} RL(\eta _ {{}_{Y}}) }}\right)\cdot 
\left( \left(RLU\mu L\right)\left(RL\eta J RL\right)\left( \rho RL\right)\right) _ {{}_{\eta _ {{}_{Y}} }}\\
&&\cdot 
\left( \id _ {{}_{RLUL(\eta _ {{}_{Y}}) }}\ast (\rrrr\llll ) ^{-1} _ {{}_{U(\mu _ {{}_{LY}})\eta _ {{}_{JRLY}} }}\ast \id _ {{}_{\rho _ {{}_{RLY}} }}\right)   
\\
\AAAA _Y (\sigma _{01})        &: =& 
\left( (\rrrr\llll )_{{}_{U(\mu _{{}_{LULY}})\eta _ {{}_{JRLULY}} }}\ast\rho _{{}_{RL(U(\mu _{{}_{LY}})\eta _ {{}_{JRLY}}) }}\ast \id _ {{}_{\rho _{{}_{RLY}} }}\right)\\
&&
\cdot   
\left(\left( (RLU\mu L)(RL\eta JRL)\right) _ {{}_{U(\mu _{{}_{LY}})\eta _ {{}_{JRLY}} }}\ast \rho _ {{}_{\rho _{{}_{RLY}} }}\right)\\
&&
\cdot 
\left( \id _ {{}_{RLUL(U(\mu _{{}_{LY}})\eta _ {{}_{JRLY}})RLU(\mu _{{}_{ERLY}}) }}\ast \left( RL\eta J\right) _{{}_{\rho _{{}_{RLY}}}}\ast \id _ {{}_{\rho _{{}_{RLY}}}}\right)\\
&&\cdot \left(  \id _ {{}_{RLUL(U(\mu _{{}_{LY}})\eta _ {{}_{JRLY}}) }}\ast  \left((\rrrr\llll\uuuu )^{-1}_{{}_{\mu _ {{}_{ERLY}}E(\rho _ {{}_{RLY}})}}\cdot RLU(v_{{}_{RLY}})\cdot (\rrrr\llll\uuuu )_ {{}_{ERLY}}\right)\ast \id _ {{}_{RL(\eta _ {{}_{JRLY}})\rho _ {{}_{RLY}} }}\right)\\
&&\cdot \left(\left( RL\eta\right) ^{-1}_{{}_{U(\mu _ {{}_{LY}})\eta _ {{}_{JRLY}}  }}\ast \id _ {{}_{\rho _ {{}_{RLY}} }}\right)
\end{eqnarray*} 
\normalsize
\end{defi}

\begin{theo}[Biadjoint Triangle]\label{MAIN}
Let $(E\dashv R, \rho , \mu , v,w)$ and $(L\dashv U, \eta , \varepsilon , s,t)$ be biadjunctions such that 
$$\xymatrix{  \aaa\ar[rr]^-{J}\ar[dr]_-{E}&&\bbb\ar[dl]^-{L}\\
&\ccc  & }$$
is a commutative triangle of pseudofunctors.
Assume that, for each pair of objects $(Y\in\bbb , A\in\aaa )$, the $2$-functor 
\small
\begin{equation*}\tag{$\DDDD_Y^{JA}$}
\xymatrix{ \bbb (JA, Y) \ar[d]|-{L_{{}_{JA,Y}}}\\  \ccc (LJA, LY)\ar@<-2.5ex>[rrr]_-{\ccc(LJA, L(\eta _{{}_Y} ))}\ar@<2.5ex>[rrr]^-{L_{{}_{{}_{JA,ULY}}}\circ\hspace{0.2em} \chi_ {{}_{{}_{(JA,LY)}}}} &&& \ccc (LJA, LULY )\ar[lll]|-{\ccc (LJA,\varepsilon _ {{}_{LY}})}
\ar@<-2.5 ex>[rrr]_-{\ccc (LJA, LUL(\eta _{{}_Y} ))}\ar[rrr]|-{\ccc (LJA, L(\eta _ {{}_{ULY}}))}\ar@<2.5ex>[rrr]^-{L_{{}_{{}_{JA,(UL)^2Y}}}\circ \hspace{0.2em}\chi_ {{}_{{}_{(JA,LULY)}}}} &&& \ccc (LJA, L(UL)^2Y) }
\end{equation*}
\normalsize
is of effective descent. The pseudofunctor $J$ has a right biadjoint if and only if, for every object $Y$ of $\bbb $, the descent object of the diagram 
$\AAAA _Y: \Delta\to \aaa$
\normalsize
exists in $\aaa$. In this case, $J$ is left biadjoint to $G$, defined by
$GY:= \left\{ \dot{\Delta } (\mathsf{0}, \j - ), \AAAA _Y\right\} _{\bi }.$
\end{theo}

\begin{proof}
We denote by $\xi : \ccc(E-, -)\simeq \aaa (-,R-)$ the pseudonatural equivalence associated to the biadjunction $(E\dashv R, \rho , \mu , v,w)$ (see Remark \ref{addd}). 		
For each object $A$ of $\aaa $ and each object $Y$ of $\bbb $, the components of $\xi $ induce a pseudonatural equivalence
$$\psi : \DDDD_Y^{JA}\circ\j \longrightarrow \aaa(A, \AAAA _Y -) $$
in which 
\begin{eqnarray*}
\psi _ {{}_{\mathsf{1}}} &:=& \xi _ {{}_{(A, LY)}}:  \ccc(EA, LY)\to \aaa (A, RLY) \\
\psi _ {{}_{\mathsf{2}}} &:=&  \xi _ {{}_{(A, LULY)}}: \ccc(EA, LULY)\to \aaa (A, RLULY) \\
\psi _ {{}_{\mathsf{3}}} &:=&  \xi _ {{}_{(A, LULULY)}}: \ccc(EA, LULULY)\to \aaa (A, RLULULY)
\end{eqnarray*}
\small
\begin{equation*}
\begin{aligned}
\left(\psi _ {{}_{s^0}}\right)_{{}_f} &:=& \rrrr_{{}_{\varepsilon _ {{}_{LY}}f}}\ast\id_{{}_{\rho _ {{}_A}}}\\
\left(\psi _ {{}_{d^1}}\right)_{{}_f} &:=& \rrrr_{{}_{L(\eta _ {{}_{Y}})f}}\ast\id_{{}_{\rho _ {{}_A}}}\\
\end{aligned}
\qquad\qquad\qquad
\begin{aligned}
\left(\psi _ {{}_{\partial ^1}}\right) _ {{}_f} &:=& \rrrr_{{}_{L(\eta _ {{}_{ULY}})f}}\ast\id_{{}_{\rho _ {{}_A}}} \\
\left( \psi _ {{}_{\partial ^2}}\right) _ {{}_f} &:=& \rrrr_{{}_{LUL(\eta _ {{}_{Y}})f}}\ast\id_{{}_{\rho _ {{}_A}}} 
\end{aligned}
\end{equation*}
\begin{eqnarray*}
\left(\psi _ {{}_{d^0}}\right)_{{}_f} &:=& \left( (\rrrr\llll )_ {{}_{U(f)\eta _ {{}_{JA}}}}\ast \id _ {{}_{\rho _ {{}_{A}}}}\right)\\
&&\cdot 
\left( \id _ {{}_{RLU(f)}}\ast\left((\rrrr\llll\uuuu )_ {{}_{EA}}^{-1}\cdot RLU(v_{{}_A}^{-1})\cdot (\rrrr\llll\uuuu )_{{}_{\mu _ {{}_{EA}}E(\rho _ {{}_{A}})}}\right)\ast \id _ {{}_{RLU(\eta _ {{}_{JA}})\rho _ {{}_{A}}}}\right)
\\
&&\cdot 
\left( \id _ {{}_{RLU(f)RL(\mu _ {{}_{EA})}}}\ast \left((RL\eta J) \rho \right)_ {{}_{\rho _ {{}_{A}}}}^{-1}\right)
\cdot \left(\left( (RLU\mu )(RL\eta JR)(\rho R)\right) ^{-1}_{{}_f}\ast \id _ {{}_{\rho _ {{}_{A}}}}\right)\\
&&
\cdot\left((\rrrr\llll )^{-1}_{{}_{U(\mu _ {{}_{LY})\eta _ {{}_{JRLY}} }}}\ast \id _ {{}_{\rho _ {{}_{RLY}}R(f)\rho _ {{}_{A}}}}\right) 
\\
\left(\psi _ {{}_{\partial ^0}}\right)_{{}_f} &:=&  
\left( (\rrrr\llll )_ {{}_{U(f)\eta _ {{}_{JA}}}}\ast \id _ {{}_{\rho _ {{}_{A}}}}\right)
\\
&&\cdot 
\left( \id _ {{}_{RLU(f)}}\ast \left((\rrrr\llll\uuuu )_ {{}_{EA}}^{-1}\cdot RLU(v_{{}_A}^{-1})\cdot (\rrrr\llll\uuuu )_{{}_{\mu _ {{}_{EA}}E(\rho _ {{}_{A}})}}\right)\ast \id _ {{}_{RLU(\eta _ {{}_{JA}})\rho _ {{}_{A}}}}\right)
\\
&&\cdot 
\left( \id _ {{}_{RLU(f)RL(\mu _ {{}_{EA})}}}\ast \left((RL\eta J) \rho \right)_ {{}_{\rho _ {{}_{A}}}}^{-1}\right)
\cdot \left(\left( (RLU\mu )(RL\eta JR)(\rho R)\right) ^{-1}_{{}_f}\ast \id _ {{}_{\rho _ {{}_{A}}}}\right)\\
&&
\cdot\left((\rrrr\llll )^{-1}_{{}_{U(\mu _ {{}_{LULY})\eta _ {{}_{JRLULY}} }}}\ast \id _ {{}_{\rho _ {{}_{RLULY}}R(f)\rho _ {{}_{A}}}}\right)
\end{eqnarray*}
\normalsize
First of all, we assume that $\DDDD_Y^{JA}$ is of effective descent for every object $A$ of $\aaa $ and every object $Y$ of $\bbb $. Then the descent object of $\DDDD_Y^{JA}\circ\j\simeq\aaa(A, \AAAA _Y -) $ is  $\DDDD_Y^{JA}\mathsf{0}$. Moreover, since this is true for all objects $A$ of $\aaa $, we conclude that the descent object of $ \YYYY\circ\AAAA _Y $ is
$\ccc(J-,Y): \aaa ^{\op}\to  \CAT $.

If, furthermore, $\aaa $ has the descent object of $\AAAA _ Y $, we get that $\YYYY \left\{ \dot{\Delta } (\mathsf{0}, \j - ), \AAAA _Y\right\} _{\bi } $ is also a descent object of $\YYYY\circ\AAAA _Y $. Therefore we get a pseudonatural equivalence
$$\ccc \left(J-, Y\right)\simeq \aaa \left( -, \left\{ \dot{\Delta } (\mathsf{0}, \j - ), \AAAA _Y\right\} _{\bi }\right) .$$
This proves that $J$ is left biadjoint to $G$, provided that  the descent object of $\AAAA _ Y $ exists for every object $Y$ of $\bbb $. 

Reciprocally, if $J$ is left biadjoint to a pseudofunctor $G$, since $\ccc (-,GY)\simeq \ccc (J-, Y) $ is the descent object of $\aaa(-, \AAAA _Y -) $, we conclude that $GY$ is the descent object of $\AAAA _ Y $.			
\end{proof}

We establish below the obvious dual version of Theorem \ref{MAIN}, which is the relevant theorem to the usual context of pseudopremonadicity~\cite{CME}. 
For being able to give such dual version, we have to employ the observations given in Remark \ref{codescent} on codescent objects. Also, if $(L\dashv U, \eta , \varepsilon , s,t)$ is a biadjunction, we need to consider its associated pseudonatural equivalence $\tau : \ccc(-, U-)\to \bbb (L-,-)$. In particular,
$$\tau _ {{}_{(X,Z)}}: \ccc(X, UZ)\to \bbb (LX,Z):\qquad\qquad f      \mapsto \varepsilon _{{}_Z}L(f);\qquad \mathfrak{m}\mapsto \id _ {{}_{\varepsilon _{{}_Z} }}\ast L(\mathfrak{m}) $$

\begin{theo}[Biadjoint Triangle]\label{dualMAIN}
Let $(E\dashv R, \rho , \mu , v,w)$ and $(L\dashv U, \eta , \varepsilon , s,t)$ be biadjunctions such that
$$\xymatrix{  \aaa\ar[rr]^-{J}\ar[dr]_-{R}&&\bbb\ar[dl]^-{U}\\
&\ccc  & }$$
is a commutative triangle of pseudofunctors.
Assume that, for each pair of objects $(Y\in\bbb , A\in\aaa )$, the $2$-functor 
$$\dot{\Delta }\to \CAT $$
\small
$$\xymatrix{ \bbb (Y, JA) \ar[d]|-{U_{{}_{Y,JA}}}\\  \ccc (UY, UJA)\ar@<2.5ex>[rrr]^-{\ccc(U(\varepsilon _{{}_A} ), UJA)}\ar@<-2.5ex>[rrr]_-{U_{{}_{{}_{LUY,JA}}}\circ\hspace{0.2em} \tau_ {{}_{{}_{(UY,JA)}}}} &&& \ccc (ULUY, UJA )\ar[lll]|-{\ccc (\eta _ {{}_{UY}}, UJA)}
\ar@<2.5 ex>[rrrr]^-{\ccc (ULU(\varepsilon _{{}_Y}),UJA)}\ar[rrrr]|-{\ccc (U(\varepsilon _{{}_{LUY}}),UJA)}\ar@<-2.5ex>[rrrr]_-{U_{{}_{{}_{(LU)^2Y, JA}}}\circ \hspace{0.2em}\tau_ {{}_{{}_{(ULUY, JA)}}}} &&&& \ccc (U(LU)^2Y, UJA) }$$
\normalsize
(with omitted $2$-cells) is of effective descent.
We have that $J$ has a left biadjoint if and only if, for every object $Y$ of $\bbb $, $\aaa$ has the codescent object of the diagram (with the obvious $2$-cells)
$$\Delta ^{\op }\to\aaa $$
\small
$$\xymatrix{  EUY\ar[rrrrr]|-{E(\eta _{{}_{UY}})} &&&&&  EULUY \ar@<-2.5ex>[lllll]_-{EU(\varepsilon _{{}_Y})}\ar@<2.5ex>[lllll]^-{\mu _ {{}_{EUY}} EU(\varepsilon _{{}_{JEUY}}L(\rho _{{}_{UY}}))}
&&&&& EULULUY\ar@<-2.5 ex>[lllll]_-{EULU(\varepsilon_{{}_{Y}})}\ar[lllll]|-{EU(\varepsilon_{{}_{LUY}})}\ar@<2.5ex>[lllll]^-{ \mu _{{}_{EULUY}} EU(\varepsilon _ {{}_{JEULUY}} L(\rho _{{}_{ULUY}}))}  }$$
\normalsize
\end{theo}

\subsection{Strict Version}\label{StrictVersion}

The techniques employed  to prove strict versions of Theorem \ref{MAIN} are virtually the same. We just need to repeat the same constructions, but, now, by means of strict descent objects and $2$-adjoints. For instance, we have:

\begin{theo}[Strict Biadjoint Triangle]\label{MAINstrict}
Let $(L\dashv U, \eta , \varepsilon , s, t)$ be a biadjunction between $2$-functors and $(E\dashv R, \rho , \mu ) $ be a $2$-adjunction such that
the triangle of $2$-functors
$$\xymatrix{  \aaa\ar[rr]^-{J}\ar[dr]_-{E}&&\bbb\ar[dl]^-{L}\\
&\ccc  & }$$
commutes and $\left(\eta J\right): J\longrightarrow UE $ is a $2$-natural transformation.
We assume that, for every pair of objects $(A\in\aaa , Y\in\bbb )$, the diagram 
$\DDDD_Y^{JA}:\dot{\Delta }\to\CAT $ 
induced by $(L\dashv U, \eta , \varepsilon , s, t)$ is of strict descent.
The $2$-functor $J$ has a right $2$-adjoint if and only if, for every object $Y$ of $\bbb $, the strict descent object of $\AAAA _Y: \Delta \to \aaa $  
exists in $\aaa$.
\end{theo}

\begin{proof}
In particular, we have the setting of Theorem \ref{MAIN}. Therefore, again, we can define 
$\psi : \DDDD_Y^{JA}\circ\j \longrightarrow \aaa(A, \AAAA _Y -) $
as it was done in the proof of Theorem \ref{MAIN}. However, since $(E\dashv R, \rho , \mu ) $ is a $2$-adjunction, $J, E, R, L, U $ are $2$-functors and $(\eta J) $ is a $2$-natural transformation, the components 
$\psi _ {{}_{d^0}}$, $\psi _ {{}_{d^1}}$, $\psi _ {{}_{s^0}}$, $\psi _ {{}_{\partial ^0}}$, $\psi _ {{}_{\partial ^1}}$, $\psi _ {{}_{\partial ^2}} $  
are identities. Thereby $\psi $ is a $2$-natural transformation. 
Moreover, since $(E\dashv R, \rho , \mu ) $ is a $2$-adjunction, $\psi $ is a pointwise isomorphism. Thus it is a $2$-natural isomorphism.

Firstly, we assume that $\DDDD_Y^{JA}$ is of strict descent for every object $A$ of $\aaa $ and every object $Y$ of $\bbb $. Then the strict descent object of $\aaa(A, \AAAA _Y -) $ is  $\DDDD_Y^{JA}\mathsf{0}$. 

If, furthermore, $\aaa $ has the strict descent object of $\AAAA _ Y$, we get a $2$-natural isomorphism
$$\ccc \left(J-, Y\right)\cong \aaa \left( -, \left\{ \dot{\Delta } (\mathsf{0}, \j - ), \AAAA _Y\right\}\right) .$$
This proves that $J$ is left $2$-adjoint, provided that  the strict descent object of $\AAAA _ Y $ exists for every object $Y$ of $\bbb $. 

Reciprocally, if $J$ is left $2$-adjoint to a $2$-functor $G$, since $\ccc (-,GY)\cong \ccc (J-, Y) $ is the strict descent object of $\aaa(-, \AAAA _Y -) $, we conclude that $GY$ is the strict descent object of $\AAAA _ Y $.
\end{proof}

\section{Pseudoprecomonadicity}\label{PRECOMONADIC}
A pseudomonad~\cite{SLACK2000, MARMOLEJO} is the same as a doctrine, whose definition can be found in page 123 of \cite{RS80}, while a pseudocomonad is the dual notion. Similarly to the $1$-dimensional case, for each pseudocomonad $\TTTTT$ on a $2$-category $\ccc $, there is an associated right biadjoint to the forgetful $2$-functor 
$\mathsf{L}: \mathsf{Ps}\textrm{-}\TTTTT\textrm{-}\CoAlg\to \ccc $, 
in which $\mathsf{Ps}$-$\TTTTT\textrm{-}\CoAlg $ is the $2$-category of pseudocoalgebras~\cite{SLACK2000} of Definition \ref{PSEUDOCOALGBRAS}. Also, every biadjunction $(L\dashv U, \eta , \varepsilon , s, t) $ induces a comparison pseudofunctor and an Eilenberg-Moore factorization~\cite{CME}    
\[\xymatrix{\bbb \ar[r]^-{\KKKK}\ar[rd]_-{L}& \mathsf{Ps} \textrm{-} \TTTTT\textrm{-}\CoAlg\ar[d]\\
&\ccc }\] 
in which $\TTTTT $ denotes the induced pseudocomonad.
Before proving Corollary \ref{PMAIN} which is a consequence of Theorem \ref{MAIN} in the context of pseudocomonads, we sketch some basic definitions and known results needed to fix notation and show Lemma \ref{DescentAlgebra}. Some of them are related to the formal theory of pseudo(co)monads developed by Lack~\cite{SLACK2000}. There, it is employed the coherence result of tricategories~\cite{PowerStreet} (and, hence, with due adaptations, the formal theory developed therein works for any tricategory).

\begin{defi}[Pseudocomonad]
A \textit{pseudocomonad}  $\TTTTT = (\TTTTT, \varpi, \varepsilon, \Lambda, \delta , \mathsf{s}  ) $ on a $2$-category $\ccc $ is a pseudofunctor $(\TTTTT , \tttt ) : \ccc\to \ccc $ with
\begin{enumerate}
\item Pseudonatural transformations:
\begin{eqnarray*}
\begin{aligned}
\varpi &:& \TTTTT \longrightarrow \TTTTT ^2
\end{aligned}
\qquad\qquad
\begin{aligned}
\varepsilon &:&\TTTTT\longrightarrow \id _ {{}_{\ccc}}
\end{aligned}
\end{eqnarray*}

\item Invertible modifications:
\begin{eqnarray*}
\Lambda &:&  (\varpi\TTTTT)(\varpi )\Longrightarrow (\TTTTT\varpi )(\varpi )\\
\mathsf{s} &:&(\varepsilon\TTTTT)(\varpi )\Longrightarrow \id _ {{}_{\TTTTT }}\\
\delta &:&\id _ {{}_{\TTTTT }}\Longrightarrow (\TTTTT\varepsilon)(\varpi )
\end{eqnarray*}
\end{enumerate}
such that the following equations hold:

\begin{itemize}
\item Associativity:
\small
$$\xymatrix{  \TTTTT\ar[d]_{\varpi }\ar[r]^{\varpi }\ar@{}[rd]|{\xLeftarrow{\Lambda }} &
\TTTTT ^2 \ar[d]^{\varpi \TTTTT }\ar[rd]^{\varpi \TTTTT } &
&
\TTTTT\ar[d]_{\varpi }\ar[r]^{\varpi }\ar[rd]|{\varpi }\ar@{}[rrd]|{\xLeftarrow{\Lambda } }\ar@{}[rdd]|{\xLeftarrow{\Lambda }}&
\TTTTT ^2\ar[rd]^{\varpi \TTTTT}&  
\\
\TTTTT ^2\ar[r]_{\TTTTT\varpi }\ar[rd]_{\TTTTT\varpi }\ar@{}[rrd]|{\xLeftarrow{\widehat{\TTTTT \Lambda }} }&
\TTTTT ^3\ar[rd]|{\TTTTT\varpi \TTTTT }\ar@{}[r]|{\xLeftarrow{\Lambda\TTTTT }}&
\TTTTT ^3\ar[d]^{\varpi \TTTTT ^2}\ar@{}[r]|{=}&
\TTTTT ^2\ar[rd]_{\TTTTT \varpi } &
\TTTTT ^2\ar[d]|{\varpi \TTTTT}\ar[r]^{\TTTTT\varpi }\ar@{}[rd]|{\xLeftarrow{\varpi _ {{}_{\varpi}}^{-1} } }&
\TTTTT ^3\ar[d] ^{\varpi \TTTTT ^2} 
\\
&
\TTTTT ^3\ar[r]_{\TTTTT ^2\varpi }&
\TTTTT ^4 &
&
\TTTTT ^3 \ar[r]_{\TTTTT ^2\varpi}&
\TTTTT ^4
}$$ 
\normalsize
\item Identity:
\small
$$\xymatrix{ &
\TTTTT \ar[dl]_-{\varpi }\ar[dr]^-{\varpi }\ar@{}[dd]|{\xLeftarrow{\Lambda }}
&
&&
&
\TTTTT \ar[d]|-{\varpi }
&
\\
\TTTTT ^2\ar[dr]_-{\TTTTT\varpi }
&
&
\TTTTT ^2\ar[dl]^-{\varpi\TTTTT }
&&
&\TTTTT ^2\ar[dl]|-{\TTTTT\varpi }\ar[dr]|-{\varpi\TTTTT  }\ar@{=}[dd]
&
\\
&
\TTTTT ^3\ar[d]|-{\TTTTT\varepsilon\TTTTT }                 
&
&=&
\TTTTT ^3\ar@{}[r]|-{\xLeftarrow{\delta\TTTTT }}\ar[dr]|-{\TTTTT\varepsilon\TTTTT }
&&
\TTTTT^3\ar[dl]|-{\TTTTT\varepsilon\TTTTT }\ar@{}[l]|-{\xLeftarrow{{\widehat{\TTTTT\mathsf{s}}} } }
\\
&
\TTTTT ^2
&
&&
&
\TTTTT ^2
&
}$$ 
\end{itemize}
in which $\widehat{\TTTTT\mathsf{s}}, \widehat{\TTTTT \Lambda }  $ denote ``corrections'' of domain and codomain given by the isomorphisms induced by the pseudofunctor $\TTTTT $. That is to say,
\begin{equation*}
\begin{aligned}
\widehat{\TTTTT\mathsf{s}} &: =& \tttt ^{-1} _ {{}_{(\varepsilon\TTTTT ) (\varpi )}}(\TTTTT\mathsf{s} )\tttt _ {{}_{\TTTTT ^2 }}
\end{aligned}
\qquad\qquad
\begin{aligned}
\widehat{\TTTTT \Lambda } &: =& \tttt ^{-1}_{{}_{(\TTTTT \varpi ) (\varpi )}}(\TTTTT \Lambda ) \tttt _ {{}_{(\varpi\TTTTT)(\varpi )}}
\end{aligned}
\end{equation*}  
\end{defi}

\begin{defi}[Pseudocoalgebras]\label{PSEUDOCOALGBRAS}
Let $\TTTTT = (\TTTTT, \varpi, \varepsilon, \Lambda, \delta , \mathsf{s}  ) $ be a pseudocomonad in $\ccc $. We define the objects, $1$-cells and $2$-cells of the $2$-category $\mathsf{Ps}\textrm{-}\TTTTT\textrm{-}\CoAlg $
as follows:
\begin{enumerate}
\item Objects: pseudocoalgebras are defined by $\mathsf{z}= (Z, \varrho _{{}_{\mathsf{z}}}, \varsigma _ {{}_{\mathsf{z}}}, \Omega_{{}_\mathsf{z}} )$ in which $\varrho _{{}_{\mathsf{z}}}: Z\to\TTTTT Z $
is a morphism in $\ccc $ and
\begin{equation*}
\begin{aligned}
\varsigma _ {{}_{\mathsf{z}}}&:& \id _ {{}_{Z}}   \Rightarrow     \varepsilon _ {{}_Z}\varrho _ {{}_{\mathsf{z}}}
\end{aligned}
\qquad\qquad
\begin{aligned}
\Omega _ {{}_{\mathsf{z}}} &:& \varpi _ {{}_{Z}} \varrho _{{}_{\mathsf{z}}}\Rightarrow \TTTTT (\varrho _{{}_{\mathsf{z}}}) \varrho _{{}_{\mathsf{z}}}
\end{aligned}
\end{equation*}
\normalsize
are invertible $2$-cells of $\ccc $
such that the equations
$$\xymatrix{  Z\ar[d]_{\varrho _ {{}_\mathsf{z}} }\ar[r]^{\varrho _ {{}_\mathsf{z}} }\ar@{}[rd]|{\xLeftarrow{\Omega _ {{}_{\mathsf{z}}} }} &
\TTTTT Z \ar[d]^{\varpi _{{}_{Z}} }\ar[rd]^{\varpi _{{}_Z} } &
&
Z\ar[d]_{\varrho_ {{}_{\mathsf{z}}} }\ar[r]^{\varrho _ {{}_{\mathsf{z}}} }\ar[rd]|{\varrho _ {{}_{\mathsf{z}}} }\ar@{}[rrd]|{\xLeftarrow{\Omega _ {{}_{\mathsf{z}}} } }\ar@{}[rdd]|{\xLeftarrow{\Omega _ {{}_{\mathsf{z}}}}}&
\TTTTT Z\ar[rd]^{\varpi _{{}_Z}}&  
\\
\TTTTT Z \ar[r]_{\TTTTT (\varrho _ {{}_{\mathsf{z}}} ) }\ar[rd]_{\TTTTT (\varrho _ {{}_{\mathsf{z}}}) }\ar@{}[rrd]|{\xLeftarrow{\widehat{\TTTTT (\Omega_ {{}_{\mathsf{z}}}) }} }&
\TTTTT ^2 Z\ar[rd]|{(\TTTTT\varpi) _ {{}_{Z}} }\ar@{}[r]|{\xLeftarrow{\Lambda _ {{}_Z} }}&
\TTTTT ^2 Z \ar[d]^{\varpi _ {{}_{ \TTTTT Z}} }\ar@{}[r]|{=}&
\TTTTT Z\ar[rd]_{\TTTTT (\varrho _ {{}_{\mathsf{z}}}) } &
\TTTTT Z \ar[d]|{\varpi _ {{}_{Z}}}\ar[r]^{\TTTTT (\varrho _ {{}_{\mathsf{z}}} ) }\ar@{}[rd]|{\xLeftarrow{\varpi _ {{}_{\varrho _ {{}_{\mathsf{z}}} }}^{-1} } }&
\TTTTT ^2 Z\ar[d] ^{\varpi _{{}_{\TTTTT Z}}} 
\\
&
\TTTTT ^2 Z \ar[r]_{\TTTTT ^2(\varrho _ {{}_{\mathsf{z}}} ) }&
\TTTTT ^3 Z &
&
\TTTTT ^2 Z \ar[r]_{\TTTTT ^2(\varrho_ {{}_{\mathsf{z}}})}&
\TTTTT ^3 Z
}$$ 
\normalsize
\small
$$\xymatrix{ &
Z \ar[dl]_-{\varrho _ {{}_{\mathsf{z}}} }\ar[dr]^-{\varrho _ {{}_{\mathsf{z}}} }\ar@{}[dd]|{\xLeftarrow{\Omega _ {{}_{\mathsf{z}}} }}
&
&&
&
Z \ar[d]|-{\varrho _ {{}_{\mathsf{z}}} }
&
\\
\TTTTT Z\ar[dr]_-{\TTTTT (\varrho _ {{}_{\mathsf{z}}}) }
&
&
\TTTTT Z \ar[dl]^-{\varpi _{{}_{Z}} }
&&
&\TTTTT Z\ar[dl]|-{\TTTTT (\varrho _ {{}_{\mathsf{z}}} ) }\ar[dr]|-{\varpi _ {{}_Z}  }\ar@{=}[dd]
&
\\
&
\TTTTT ^2 Z \ar[d]|-{(\TTTTT\varepsilon ) _ {{}_{Z}} }                 
&
&=&
\TTTTT ^2 Z \ar@{}[r]|-{\xLeftarrow{\widehat{\TTTTT (\varsigma _ {{}_{\mathsf{z}}})}} }\ar[dr]|-{(\TTTTT\varepsilon ) _ {{}_{Z}} }
&&
\TTTTT^2 Z \ar[dl]|-{(\TTTTT \varepsilon ) _{{}_{Z}} }\ar@{}[l]|-{\xLeftarrow{\delta _ {{}_{Z}} } }
\\
&
\TTTTT Z
&
&&
&
\TTTTT Z
&
}$$ 
\normalsize
are satisfied, in which
\small 
\begin{equation*}
\begin{aligned}
\widehat{\TTTTT (\varsigma _ {{}_{\mathsf{z}}})} &:= & \tttt ^{-1} _ {{}_{\varepsilon _ {{}_{Z}} \varrho _ {{}_{\mathsf{z} }}  }} \TTTTT (\varsigma _ {{}_{\mathsf{z}}})\tttt _ {{}_{Z }}
\end{aligned}
\qquad\qquad
\begin{aligned}
\widehat{\TTTTT (\Omega_ {{}_{\mathsf{z}}}) } &:= &\tttt ^{-1} _ {{}_{\varpi _ {{}_{Z}} \varrho _ {{}_{\mathsf{z} }}   }} \TTTTT (\Omega_ {{}_{\mathsf{z}}})\tttt _ {{}_{\TTTTT (\varrho _ {{}_{\mathsf{z}}} )\varrho _ {{}_{\mathsf{z} }} }} 
\end{aligned}
\end{equation*}
\normalsize

\item Morphisms: $\TTTTT$-pseudomorphisms $\mathsf{f}:\mathsf{x}\to\mathsf{z} $ are pairs $\mathsf{f} = (f, \varrho _ {{}_{f}} ^{-1} ) $ in which 
$f: X\to Z  $
is a morphism in $\ccc $ and
$\varrho _ {{}_{\mathsf{f}}}: \TTTTT (f) \varrho _{{}_{\mathsf{x} }} \Rightarrow \varrho _{{}_{\mathsf{z} }}f   $  
is an invertible $2$-cell of $\ccc $ such that, defining $\widehat{\TTTTT (\varrho _ {{}_{\mathsf{f}}} ^{-1})} : = \tttt ^{-1} _ {{}_{\TTTTT (f)  \varrho _ {{}_{\mathsf{x} }} }} \TTTTT (\varrho _ {{}_{\mathsf{f}}} ^{-1})\tttt _ {{}_{\varrho _ {{}_{\mathsf{z}}}f  }}$, 
\small
$$\xymatrix{  X\ar[d]_{\varrho _ {{}_{\mathsf{x}}}}\ar[r]^{f}\ar@{}[rd]|{\xLeftarrow{\varrho _ {{}_{\mathsf{f}}} ^{-1} }} &
Z \ar[d]^{\varrho _ {{}_{\mathsf{z}}}}\ar[rd]^{\varrho _ {{}_{\mathsf{z}}}} &
&
X\ar[d]_{\varrho _ {{}_{\mathsf{x}}} }\ar[r]^{f}\ar[rd]|{\varrho _ {{}_{\mathsf{x}}} }\ar@{}[rrd]|{\xLeftarrow{\varrho _ {{}_{\mathsf{f} }} ^{-1} } }\ar@{}[rdd]|{\xLeftarrow{\Omega _ {{}_{\mathsf{x}}} }}&
Z\ar[rd]^{\varrho_{{}_{\mathsf{z}}} }&  
\\
\TTTTT X\ar[r]_{\TTTTT (f)}\ar[rd]_{\TTTTT(\varrho _ {{}_{\mathsf{x}}} )}\ar@{}[rrd]|{\xLeftarrow{\widehat{\TTTTT (\varrho _ {{}_{\mathsf{f}}} ^{-1})}}}&
\TTTTT Z\ar[rd]|{\TTTTT (\varrho _ {{}_{\mathsf{z}}} ) }\ar@{}[r]|{\xLeftarrow{\Omega _ {{}_{\mathsf{z}}} } }&
\TTTTT Z \ar[d]^{\varpi _ {{}_{Z}} }\ar@{}[r]|{=}&
\TTTTT X \ar[rd]_{\TTTTT (\varrho _ {{}_{\mathsf{x} }} ) } &
\TTTTT X\ar[d]|{\varpi _{{}_{X}} }\ar[r]^{\TTTTT (f)}\ar@{}[rd]|{\xLeftarrow{ \varpi _ {{}_{f}}^{-1}  } }&
\TTTTT Z \ar[d] ^{\varpi _ {{}_{Z}} } 
\\
&
\TTTTT ^2 X\ar[r]_{\TTTTT ^2 (f)}&
\TTTTT ^2 Z &
&
\TTTTT ^2 X \ar[r]_{\TTTTT ^2 (f) }&
\TTTTT^2 Z  
}$$ 
\normalsize
holds and the $2$-cell below is the identity.
\small
$$\xymatrix{ 
X\ar[rrrr]|{f}\ar[rd]|{\varrho _ {{}_{\mathsf{x}}}}\ar@{=}[dd]
&&&&
Z\ar[ld]|{\varrho _ {{}_{\mathsf{z}}} }\ar@{=}[dd]\\
&
\TTTTT X\ar@{}[l]|{\xLeftarrow{\varsigma _ {{}_{\mathsf{x}}} }}\ar[rr]|{\TTTTT (f)}\ar[ld]|{\varepsilon _ {{}_{X}} }
&
\ar@{}[u]|{\xLeftarrow{\varrho _ {{}_{\mathsf{f}}} }}\ar@{}[d]|{\xLeftarrow{(\varepsilon ) _ {{}_f}^{-1}}}
&
\TTTTT Z \ar@{}[r]|{\xLeftarrow{\varsigma _ {{}_{\mathsf{z}}}^{-1}}}\ar[rd]|{\varepsilon _ {{}_{Z}}}
&
\\
X\ar[rrrr]|{f}
&&&&
Y  
}$$
\normalsize
\item $2$-cells: a $\TTTTT $-transformation between $\TTTTT $-pseudomorphisms $\mathsf{m} : \mathsf{f}\Rightarrow\mathsf{h} $ is a $2$-cell $\mathfrak{m} : f\Rightarrow h $ in $\ccc $ such that  the equation below holds.

\tiny
$$\xymatrix{  X\ar@/_6ex/[dd]_{f }
                    \ar@{}[dd]|{\xRightarrow{\mathfrak{m} } }
                    \ar@/^6ex/[dd]^{h }
										\ar[rr]^{\varrho _ {{}_{\mathsf{x}}} } && \TTTTT X\ar[dd]^{\TTTTT (h) }    
&&
X\ar[rr]^{ \varrho _ {{}_{\mathsf{x}}} }\ar[dd]_{f }  &&\TTTTT X\ar@/_6ex/[dd]_{\TTTTT (f) }
                    \ar@{}[dd]|{\xRightarrow{\TTTTT (\mathfrak{m} ) } }
                    \ar@/^6ex/[dd]^{\TTTTT (h) }
\\
&\ar@{}[r]|{\xRightarrow{\hskip .2em \varrho _{{}_{h}}  \hskip .2em } }   &
 &=& 
&\ar@{}[l]|{\xRightarrow{\hskip .2em \varrho _{{}_{f}}  \hskip .2em } }  & 
\\
 Z\ar[rr]_ {\varrho _ {{}_{\mathsf{z}}} } && \TTTTT Z
 &&
Z\ar[rr]_ {\varrho _ {{}_{\mathsf{z}}} } &&\TTTTT Z	 }$$ 
\normalsize

\end{enumerate}
\end{defi}

\begin{rem}
If $\TTTTT = (\TTTTT, \varpi, \varepsilon, \Lambda, \delta , \mathsf{s}  ) $  is a pseudocomonad on $\ccc $, 
then $\TTTTT $ induces a biadjunction $(\mathsf{L}\dashv \mathsf{U}, \varrho, \varepsilon , \underline{\mathsf{s}}, \underline{\mathsf{t} })$
in which $\mathsf{L}, \mathsf{U} $ are defined by
\small
\begin{eqnarray*}
\begin{aligned}
\mathsf{L}: \mathsf{Ps}\textrm{-}\TTTTT\textrm{-}\CoAlg & \to \ccc\\ 
\mathsf{z}= (Z, \varrho _{{}_{\mathsf{z}}}, \varsigma _ {{}_{\mathsf{z}}}, \Omega_{{}_\mathsf{z}} )&\mapsto Z\\
\mathsf{f} = (f, \varrho _ {{}_{f}} ^{-1} )&\mapsto f\\
\mathsf{m} &\mapsto \mathfrak{m}
\end{aligned}
\qquad\qquad
\begin{aligned}
\mathsf{U} : \ccc &\to \mathsf{Ps}\textrm{-}\TTTTT\textrm{-}\CoAlg\\
             Z &\mapsto \left( \TTTTT (Z), \varpi _ {{}_{Z}}, \mathsf{s} _ {{}_{Z}}, \Lambda _ {{}_{Z}}\right)\\
						f &\mapsto \left(\TTTTT(f), \varpi _ {{}_{f}}^{-1}\right)\\
						\mathfrak{m} &\mapsto \TTTTT(\mathfrak{m})						
\end{aligned}
\end{eqnarray*}
\normalsize
Reciprocally, we know that each biadjunction $(L\dashv U, \eta , \varepsilon , s, t) $  induces a pseudocomonad
$$\TTTTT = (LU, L\eta U, \varepsilon, (L\eta ) _ {{}_{\eta _ {{}_{U}} }}^{-1} , \widehat{\left(Lt\right)}, sU  ) $$
Lemma \ref{basicstuff} gives some further aspects of these constructions (which follows from calculations on the formal theory of pseudocomonads in $2$-$\CAT $).

\end{rem}

\begin{lem}\label{basicstuff}
Let $L: \bbb\to\ccc $ be a pseudofunctor. A biadjunction $(L\dashv U, \eta , \varepsilon , s, t)$ induces commutative triangles 
\small
\[\xymatrix{\bbb \ar[r]^-{\KKKK}\ar[rd]_-{L}& \mathsf{Ps} \textrm{-} \TTTTT\textrm{-}\CoAlg \ar[d]^{\mathsf{L} } &&\ccc \ar[r]^-{U}\ar[rd]_-{\mathsf{U}}& \bbb \ar[d]^{\KKKK }\\
&\ccc &&& \mathsf{Ps} \textrm{-} \TTTTT\textrm{-}\CoAlg  }\] 
\normalsize
in which $\TTTTT = (\TTTTT, \varpi, \varepsilon, \Lambda, \delta , \mathsf{s}  ) $ is the pseudocomonad induced by $(L\dashv U, \eta , \varepsilon , s, t)$, $(\mathsf{L}\dashv \mathsf{U}, \varrho, \varepsilon , \underline{\mathsf{s}}, \underline{\mathsf{t} })$ is the biadjunction induced by $\TTTTT $ and $\KKKK : \bbb\to \mathsf{Ps} \textrm{-} \TTTTT\textrm{-}\CoAlg $ is the unique (up to pseudonatural isomorphism) comparison pseudofunctor making the triangles above commutative. Namely, 
\small
\begin{eqnarray*}
\KKKK : &\bbb &\to \mathsf{Ps} \textrm{-} \TTTTT\textrm{-}\CoAlg\\
        &Y &\mapsto \left(LY, L(\eta _ {{}_{Y}}), s_ {{}_{Y}}^{-1}, \left(L\eta\right) _ {{}_{\eta _ {{}_{Y}}  } }^{-1}     \right)\\
				&g &\mapsto \left(L(g), \left( L\eta\right) _{{}_{g}}^{-1}\right)\\
				&\mathfrak{m}&\mapsto L(\mathfrak{m})
\end{eqnarray*}
\normalsize
Furthermore, we have the obvious equalities
\small
\begin{equation*}
\begin{aligned}
L(\eta _ {{}_{Y}}) &=&\varrho _ {{}_{\KKKK Y}}
\end{aligned}
\qquad\qquad
\begin{aligned}
\varpi _ {{}_{ LY }} &=& (L\eta U) _ {{}_{LY}}. 
\end{aligned}
\end{equation*}
\normalsize
\end{lem}

\begin{prop}\label{diagrampseudocomonad}
Let $\TTTTT = (\TTTTT, \varpi, \varepsilon, \Lambda, \delta , \mathsf{s}  ) $ be a pseudocomonad on $\ccc $. Given $\TTTTT$-pseudocoalgebras 
$$\mathsf{x}= (X, \varrho _{{}_{\mathsf{x}}}, \varsigma _ {{}_{\mathsf{x}}}, \Omega_{{}_\mathsf{x}} ), \mathsf{z}= (Z, \varrho _{{}_{\mathsf{z}}}, \varsigma _ {{}_{\mathsf{z}}}, \Omega_{{}_\mathsf{z}} ),$$
the category $\mathsf{Ps}\textrm{-}\TTTTT\textrm{-}\CoAlg (\mathsf{x}, \mathsf{z}) $ is the strict descent object  of the diagram
$\mathbb{T}_{\mathsf{z}}^{\mathsf{x}}: \Delta\to \CAT $
\begin{equation*}\tag{$\mathbb{T}_{\mathsf{z}}^{\mathsf{x}}$}
\xymatrix{  \ccc (\mathsf{L}\mathsf{x}, \mathsf{L}\mathsf{z})\ar@<-2.5ex>[rrr]_-{\ccc(\mathsf{L}\mathsf{x}, \varrho _ {{}_{\mathsf{z}}} )}\ar@<2.5ex>[rrr]^-{\ccc(\varrho _ {{}_\mathsf{x}}, \TTTTT\mathsf{L} \mathsf{z} )\circ \hspace{0.2em}\TTTTT _ {{}_{(\mathsf{L}\mathsf{x},\mathsf{L}\mathsf{z})}} } &&& \ccc (\mathsf{L}\mathsf{x}, \TTTTT \mathsf{L}\mathsf{z} )\ar[lll]|-{\ccc (\mathsf{L}\mathsf{x},\varepsilon _ {{}_{\mathsf{L}\mathsf{z} }})}
\ar@<-2.5 ex>[rrr]_-{\ccc (\mathsf{L}\mathsf{x}, \TTTTT (\varrho _{{}_\mathsf{z}}) )}\ar[rrr]|-{\ccc (\mathsf{L}\mathsf{x}, \varpi _ {{}_{\mathsf{L}\mathsf{z}}})}\ar@<2.5ex>[rrr]^-{\ccc(\varrho _ {{}_\mathsf{x}}, \TTTTT\mathsf{L} \mathsf{z} )\circ \hspace{0.2em}\TTTTT _ {{}_{(\mathsf{L}\mathsf{x},\TTTTT \mathsf{L}\mathsf{z})}}} &&& \ccc (\mathsf{L}\mathsf{x}, \TTTTT^2 \mathsf{L}\mathsf{z} ) }
\end{equation*}
such that
\small
\begin{eqnarray*}
\begin{aligned}
&\mathbb{T}_{\mathsf{z}}^{\mathsf{x}}(\sigma _ {02}) _{{}_{f}} &:=& \left( \tttt _ {{}_{\varrho _ {{}_{\mathsf{z}}}f }}\ast \id _ {{}_{ \varrho _ {{}_{ \mathsf{x} }}  }}\right)   \\
&\mathbb{T}_{\mathsf{z}}^{\mathsf{x}}(\sigma _ {12}) _{{}_{f}} &:=&  \left( \Omega _ {{}_{\mathsf{z}}}^{-1}\ast\id_ {{}_{f}}\right)\\
&\mathbb{T}_{\mathsf{z}}^{\mathsf{x}}(n _ {1}) _{{}_{f}} &:=& \left(\varsigma _ {{}_{\mathsf{z}}}^{-1}\ast \id _ {{}_{f}}\right)
\end{aligned}
\qquad\qquad
\begin{aligned}
&\mathbb{T}_{\mathsf{z}}^{\mathsf{x}}(\sigma _ {01}) _{{}_{f}} &:=& \left(\tttt _ {{}_{\TTTTT(f) \varrho _ {{}_{\mathsf{x} }} }}\ast \id _ {{}_{\varrho _ {{}_{\mathsf{x}}} }} \right)\cdot  \left( \id _ {{}_{\TTTTT ^2(f)  }}\ast \Omega _{{}_{\mathsf{x} }}\right)\cdot \left( \varpi _ {{}_{f}}^{-1}\ast \id _ {{}_{ \varrho _ {{}_{\mathsf{x} }} }}\right)  \\
&\mathbb{T}_{\mathsf{z}}^{\mathsf{x}}(n _ {0}) _{{}_{f}} &:=& \left(\id _ {{}_{f}}\ast \varsigma _ {{}_{\mathsf{x}}}\right)\cdot 
\left(\varepsilon ^{-1}_{{}_{f}}\ast\id _ {{}_{\varrho _ {{}_{\mathsf{x} }}  }}\right) 
\end{aligned}
\end{eqnarray*}
\normalsize
\end{prop}
\begin{proof}
It follows from Definition \ref{PSEUDOCOALGBRAS} and Remark \ref{strictdescentstrict}.
\end{proof}

Recall that every biadjunction induces diagrams $\DDDD _ Y^X: \dot{\Delta } \to \CAT $  (Definition \ref{maindiagram}). Also,  for every pseudocomonad $\TTTTT $ and objects $\mathsf{x}, \mathsf{z} $ of $\mathsf{Ps}\textrm{-}\TTTTT\textrm{-}\CoAlg $, we defined in Proposition \ref{diagrampseudocomonad} a diagram $\mathbb{T}_{\mathsf{z}}^{\mathsf{x}}: \Delta\to\CAT $ whose strict descent object is $\mathsf{Ps}\textrm{-}\TTTTT\textrm{-}\CoAlg (\mathsf{x}, \mathsf{z}) $. Now,  we give the relation between these two diagrams.

\begin{lem}\label{DescentAlgebra}
Let $L: \bbb\to\ccc $ be a pseudofunctor, $(L\dashv U, \eta , \varepsilon , s, t) $  be a biadjunction and $\TTTTT = (\TTTTT, \varpi, \varepsilon, \Lambda, \delta , \mathsf{s}  ) $ be the induced pseudocomonad. For each pair $(X,Y)$ of objects in $\bbb $, $(L\dashv U, \eta , \varepsilon , s, t) $  induces the diagram
$\DDDD ^X_Y : \dot{\Delta } \to\CAT $
and $\TTTTT$ induces the diagram
$\mathbb{T}^{\KKKK X}_{\KKKK Y}: \Delta\to\CAT $
defined in Proposition \ref{diagrampseudocomonad}, in which $\KKKK : \bbb\to\mathsf{Ps}\textrm{-}\TTTTT\textrm{-}\CoAlg$ is the comparison pseudofunctor. 
In this setting, there is a pseudonatural isomorphism 
$\beta: \DDDD ^X_Y\circ\j\longrightarrow \mathbb{T}^{\KKKK X}_{\KKKK Y} $ 
for every such pair $(X,Y)$ of objects in $\bbb $. Moreover, if $L$ is a $2$-functor, $\beta $ is actually a $2$-natural isomorphism.
\end{lem} 

\begin{proof}
We can write $\mathbb{T}^{\KKKK X}_{\KKKK Y}: \Delta\to\CAT $ as follows
\small
\[\xymatrix{  \ccc (LX, LY)\ar@<-2.5ex>[rrr]_-{\ccc(LX, L(\eta _ {{}_{Y}}) )}\ar@<2.5ex>[rrr]^-{\ccc(L(\eta _ {{}_{X}}), LUL Y )\circ \hspace{0.2em}(LU) _ {{}_{L X, L Y}} } &&& \ccc (L X,  LUL Y )\ar[lll]|-{\ccc ( LX,\varepsilon _ {{}_{LY }})}
\ar@<-2.5 ex>[rrrr]_-{\ccc ( LX, LUL(\eta _ {{}_{Y}}) )}\ar[rrrr]|-{\ccc (LX, L(\eta _ {{}_{LY}}))}\ar@<2.5ex>[rrrr]^-{\ccc(L(\eta _ {{}_{X}}), LUL Y )\circ \hspace{0.2em}(LU) _ {{}_{LX, LUL Y}}} &&&& \ccc (L X, LULUL Y ) }\]
\normalsize
Furthermore, by Lemma \ref{basicstuff} and the observations given in this section, we can define a pseudonatural isomorphism
$$\beta: \DDDD ^X_Y\circ\j\longrightarrow \mathbb{T}^{\KKKK X}_{\KKKK Y} $$
such that
$\beta _ {{}_{\mathsf{1} }},\beta _ {{}_{\mathsf{2} }},\beta _ {{}_{\mathsf{3} }}$
are identity functors,
$\beta _ {{}_{d^1 }},\beta _ {{}_{\partial ^1 }},\beta _ {{}_{\partial ^2 }}, \beta _ {{}_{s^0}} $
are identity natural transformations,
$\left(\beta _ {{}_{d^0 }}\right) _ {{}_{f}}:= \llll _ {{}_{U(f)\eta _ {{}_{X}} }}$ and 
$\left( \beta _ {{}_{\partial ^0 }}\right) _ {{}_{f}} := \llll _ {{}_{U(f)\eta _ {{}_{X}} }}$.
This completes the proof.
\end{proof}

Let $(L\dashv U, \eta , \varepsilon , s, t) $ be a biadjunction and $\TTTTT $ be the induced pseudocomonad. By Lemma \ref{strictequivalence}, Proposition \ref{diagrampseudocomonad}  and Lemma \ref{DescentAlgebra}, $\mathsf{Ps}\textrm{-}\TTTTT\textrm{-}\CoAlg (\KKKK X, \KKKK Y) $ is a descent object of $\DDDD ^X_Y\circ\j $ for every pair of objects $(X,Y)$ of $\bbb $. Moreover, 
$\KKKK _{{}_{X,Y}}:\bbb (X,Y)\to  \mathsf{Ps}\textrm{-}\TTTTT\textrm{-}\CoAlg (\KKKK X, \KKKK Y) $
is the comparison $\DDDD ^X_Y\mathsf{0}\to \left\{ \dot{\Delta }(\mathsf{0}, \j - ), \DDDD ^X_Y\circ\j \right\} $. Thereby we get:

\begin{prop}\label{prepseudocomonadicfundamental}
Let $(L\dashv U, \eta , \varepsilon , s, t) $ be a biadjunction, $\TTTTT $ be the induced pseudocomonad and $\KKKK : \bbb\to  \mathsf{Ps}\textrm{-}\TTTTT\textrm{-}\CoAlg  $ be the comparison pseudofunctor. For each pair of objects $(X,Y) $ in $\bbb $,   $\DDDD ^X_Y : \dot{\Delta } \to\CAT$ is of effective descent if and only if 
$$\KKKK _{{}_{X,Y}}:\bbb (X,Y)\to  \mathsf{Ps}\textrm{-}\TTTTT\textrm{-}\CoAlg (\KKKK X, \KKKK Y) $$
is an equivalence. Furthermore, if $L$ is a $2$-functor, $\DDDD ^X_Y$ is of strict descent if and only if
$\KKKK _{{}_{X,Y}}$
is an isomorphism.
\end{prop}

\subsection{Biadjoint Triangles}\label{refined}

In this subsection, we reexamine the results of Section \ref{Main} in the context of pseudocomonad theory. More precisely, we prove Corollary \ref{PMAIN} of our main theorems in Section \ref{Main}, Theorem \ref{MAIN} and Theorem \ref{MAINstrict}. 

Let $(L,U, \eta , \varepsilon , s, t) $ be a biadjunction and $\TTTTT $ be its induced pseudocomonad. We say that $L: \bbb\to\ccc $ is \textit{pseudoprecomonadic}, if its induced comparison pseudofunctor $\KKKK:\bbb\to \mathsf{Ps}\textrm{-}\TTTTT\textrm{-}\CoAlg $ is locally an equivalence. As a consequence of Proposition \ref{prepseudocomonadicfundamental}, we get a characterization of pseudoprecomonadic pseudofunctors.

\begin{coro}[Pseudoprecomonadic]\label{preprepseudocomonadic}
Let $(L\dashv U, \eta , \varepsilon , s, t) $ be a biadjunction. The pseudofunctor $L: \bbb\to\ccc$ is pseudoprecomonadic if and only if $\DDDD ^X_Y: \dot{\Delta } \to\CAT $
 is of effective descent for every pair of objects $(X,Y) $ of $\bbb $.
\end{coro}

By Corollary \ref{preprepseudocomonadic}, assuming that $(L\dashv U, \eta , \varepsilon , s, t) $ is a biadjunction and $J: \aaa\to\bbb $ is a pseudofunctor, if $L: \bbb\to\ccc $ is pseudoprecomonadic, then, in particular, 
$\DDDD_Y^{JA}:\dot{\Delta }\to\CAT $ 
 is of effective descent for every object $A$ of $\aaa $ and every object $Y$ of $\bbb $. Thereby, as a consequence of Theorem \ref{MAIN}, Theorem \ref{MAINstrict} and Propostion \ref{prepseudocomonadicfundamental}, we get:

\begin{coro}[Biadjoint Triangle Theorem]\label{PMAIN}
Assume that $(E\dashv R, \rho , \mu , v,w), (L\dashv U, \eta , \varepsilon , s,t)$ are biadjunctions such that the triangle of pseudofunctors
\small
\[\xymatrix{  \aaa\ar[rr]^-{J}\ar[dr]_-{E}&&\bbb\ar[dl]^-{L}\\
&\ccc & }\]
\normalsize
is commutative and $L$ is pseudoprecomonadic. Then $J$ has a right biadjoint if and only if, for every object $Y$ of $\bbb $, $\aaa$ has the descent object of the diagram $\AAAA _Y : \Delta \to \aaa $. In this case, $J$ is left biadjoint to 
$GY:= \left\{ \dot{\Delta } (\mathsf{0}, \j - ), \AAAA _Y\right\} _{\bi }$.

If, furthermore, $E, R, J, L, U$ are $2$-functors, $(E\dashv R, \rho , \mu ) $ is a $2$-adjunction, $\left(\eta J\right)$ is a $2$-natural transformation and the comparison $2$-functor $\KKKK : \bbb\to \mathsf{Ps}\textrm{-}\TTTTT\textrm{-}\CoAlg $ induced by the biadjunction $L\dashv U $ is locally an isomorphism, then
$J$ is left $2$-adjoint if and only if the strict descent object of $\AAAA _Y $ exists for every object $Y$ of $\bbb $. In this case, 
$GY:= \left\{ \dot{\Delta } (\mathsf{0}, \j - ), \AAAA _Y\right\} $
defines the right $2$-adjoint to $J$.
\end{coro}

\section{Unit and Counit}\label{counitunit}
In this section, we show that the pseudoprecomonadicity characterization given in Theorem 3.5 of \cite{CME} is a consequence of Corollary \ref{preprepseudocomonadic}. Secondly, we study again biadjoint triangles. Namely, in the context of Corollary \ref{PMAIN}, we give necessary and sufficient conditions under which the unit and the counit of the obtained biadjunction $J\dashv G $ are pseudonatural equivalences, provided that $E$ and $L$ induce the same pseudocomonad. In other words, we prove the appropriate analogous versions of Corollary 1 and Corollary 2 of page 76 in \cite{Dubuc} within our context of biadjoint triangles. 

Again, we need to consider another type of $2$-functors induced by biadjunctions. The definition below is given in Theorem 3.5 of \cite{CME}.
 
\begin{defi}\label{almostabsolutediagram}
Assume that $L:\bbb\to\ccc $ is a pseudofunctor and   $(L\dashv U, \eta , \varepsilon , s, t)$ is a biadjunction. For each object $Y$ of $\bbb $, we get the $2$-functor
$\VVVV _Y : \dot{\Delta}\to \bbb $
\small
\begin{equation*}\tag{$\VVVV _ Y $}
\xymatrix{  Y \ar[r]^-{\eta _ {{}_{Y}}} & ULY\ar@<-2.5ex>[rrr]_-{UL(\eta _{{}_Y})}\ar@<2.5ex>[rrr]^-{\eta_ {{}_{ULY}}} &&& ULULY\ar[lll]|-{U(\varepsilon _ {{}_{LY}})}
\ar@<-2.5 ex>[rrr]_-{ULUL(\eta _ {{}_{Y}})}\ar[rrr]|-{UL(\eta _ {{}_{ULY}})}\ar@<2.5ex>[rrr]^-{\eta_ {{}_{ULULY}}} &&& ULULULY }
\end{equation*}
\normalsize
in which 
$\displaystyle\VVVV _Y (\vartheta ) := \left( \eta _ {{}_{\eta _{{}_Y}}}:UL(\eta _{{}_Y})\eta _ {{}_Y}\cong  \eta _ {{}_{ULY}}\eta _{{}_Y}\right) $ 
is the invertible $2$-cell component of the unit $\eta $ at the morphism $\eta _ {{}_Y} $. Analogously, the images of the $2$-cells $\sigma _ {ik}, n_0, n_1$ are defined below.
\small
\begin{equation*}
\begin{aligned}
\VVVV _Y (\sigma _{01})        &:=   \eta _ {{}_{\eta _{{}_{ULY}}}}\\ 
 \VVVV _Y (\sigma _{02})        &:=   \eta _ {{}_{UL(\eta _{{}_Y})}}\\
\VVVV _Y  (\sigma _{12})       &:=   (UL\eta) _ {{}_{\eta _{{}_Y}}}
\end{aligned}
\qquad\qquad
\begin{aligned}
\VVVV _Y  (n_0) & :=  t_{{}_{LY}}\\
\VVVV _Y (n_1) & := \left(\uuuu_{{}_{\varepsilon _ {{}_{LY}},L(\eta _{{}_Y})}}\right)^{-1}\cdot U(s_{{}_Y})\cdot \uuuu _ {{}_{ULY}}
\end{aligned}
\end{equation*}  
\normalsize
\end{defi}
We verify below that \textit{$ \VVVV _Y $ is well defined}. That is to say, we have to prove that $ \VVVV _Y $ satisfies the equations given in Definition \ref{Presentation}. Firstly, the associativity and naturality equations of Definition \ref{pseudonaturaltransformation} give the following equality
\small
$$\xymatrix{ &&
\VVVV _Y\mathsf{0}\ar[rr]|-{\VVVV _Y(d)}\ar[d]|-{\VVVV _Y(d)}\ar[lld]|-{\VVVV _Y(d)}\ar@{}[rrd]|-{\xLeftarrow{\VVVV _Y(\vartheta ) } }
&&
\VVVV _Y\mathsf{1}\ar[d]|-{\VVVV _Y(d^1)}\ar@{}[rrdd]|-{=}
&& 
\VVVV _Y\mathsf{0}\ar[rr]|-{\VVVV _Y(d)}\ar@{}[rrd]|-{\xLeftarrow{\VVVV _Y(\vartheta )}}\ar[d]|-{\VVVV _Y(d)}
&&
\VVVV _Y\mathsf{1}\ar[d]|-{\VVVV _Y (d^1)}\ar[rrd]|-{\VVVV _Y(d^1)}
&&
\\
\VVVV _Y\mathsf{1}\ar@{}[rr]|-{\xLeftarrow{\VVVV _Y(\vartheta ) } }\ar[rrd]|-{\VVVV _Y(d^0)}
&&
\VVVV _Y\mathsf{1}\ar[rr]|-{\VVVV _Y(d^0)}\ar[d]|-{\VVVV _Y(d^1)}\ar@{}[rrd]|-{\xLeftarrow{\VVVV _Y(\sigma _ {02}) }}
&&
\VVVV _Y\mathsf{2}\ar[d]|-{\VVVV _Y(\partial ^2) }
&&
\VVVV _Y\mathsf{1}\ar[rr]|-{\VVVV _Y(d^0)}\ar[d]|-{\VVVV _Y(d^0)}\ar@{}[rrd]|-{\xLeftarrow{\VVVV _Y(\sigma _ {01}) } }
&&
\VVVV _Y\mathsf{2}\ar@{}[rr]|-{\xLeftarrow{\VVVV _Y(\sigma _ {12}) } }\ar[d]|-{\VVVV _Y(\partial ^1)}
&&
\VVVV _Y\mathsf{2}\ar[lld]|-{\VVVV _Y(\partial ^2)}\\
&&
\VVVV _Y\mathsf{2}\ar[rr]|-{\VVVV _Y(\partial ^0)}
&&
\VVVV _Y\mathsf{3}
&&
\VVVV _Y\mathsf{2}\ar[rr]|- {\VVVV _Y(\partial ^0)}
&&
\VVVV _Y\mathsf{3}
}$$
\normalsize
which is the associativity equation of Definition \ref{Presentation}. Furthermore, by Definition \ref{ad} of biadjunction, we have that
\small
$$\xymatrix{  \VVVV _Y\mathsf{0}     \ar[rr]^-{\VVVV _Y(d)=\eta_ {{}_Y}} 
                    \ar[dd]|-{\VVVV _Y(d)=\eta_ {{}_Y}} 
                      && 
              {\VVVV _Y\mathsf{1}}  \ar[dd]|-{\VVVV _Y(d^1)=UL(\eta _ {{}_Y})}
                    \ar@{=}@/^5ex/[dddr]
                    \ar@{}[dddr]|-{\xLeftarrow{\VVVV _Y(n_1)}} &&&
        \VVVV _Y\mathsf{0}    \ar@/_4ex/[ddd]_-{\VVVV _Y(d)}
                    \ar@{}[ddd]|=
                    \ar@/^4ex/[ddd]^-{\VVVV _Y(d)}
                                        \\ 
          & \xLeftarrow{\hskip .1em \VVVV _Y(\vartheta )}  && &=                            
                                        \\
              {\VVVV _Y\mathsf{1})}  \ar[rr]_{\VVVV _Y(d^0)=\eta _ {{}_{ULY}}}
                    \ar@{=}@/_5ex/[drrr]
                    \ar@{}[drrr]|{\xLeftarrow{\VVVV _Y(n_0)}} &&
              {\VVVV _Y\mathsf{2}}\ar[dr]|-{\VVVV _Y(s^0)}  \\                                            
          &&& {\VVVV _Y\mathsf{1}} && {\VVVV _Y\mathsf{1}} }$$					
\normalsize					
which proves that $\VVVV _Y$ satisfies the identity equation of Definition \ref{Presentation}.

As mentioned before, Corollary \ref{CMEprecom} is Theorem 3.5 of \cite{CME}. Below, it is proved as a consequence of Corollary \ref{preprepseudocomonadic}.

\begin{coro}[\cite{CME}]\label{CMEprecom}
Let $(L\dashv U, \eta , \varepsilon , s, t)$ be a biadjunction. The pseudofunctor $L$ is pseudoprecomonadic if and only if, for every object $Y$ of $\bbb $, the $2$-functor $\VVVV _Y : \dot{\Delta}\to \bbb $ is of effective descent. 
\end{coro}
\begin{proof}
On one hand, by Corollary \ref{preprepseudocomonadic}, $L$ is pseudoprecomonadic if and only if $\DDDD ^X_Y:\dot{\Delta }\to\CAT $  is of effective descent for every pair $(X,Y)$ of objects in $\bbb $. On the other hand, by Lemma \ref{Yonchange}, $\VVVV _Y $ is of effective descent if and only if 
$\bbb (X, \VVVV _Y-): \dot{\Delta}\to\CAT $
is of effective descent for every object $X$ in $\bbb $.

Therefore, by Lemma \ref{triiviall}, to complete our proof, we just need to verify that $ \DDDD ^X_Y\simeq \bbb (X, \VVVV _Y-)$. Indeed, there is a pseudonatural equivalence
$$\iota : \DDDD ^X_Y\longrightarrow \bbb (X, \VVVV _Y-) $$
induced by $\chi : \ccc(L-,-)\simeq\bbb (-, U-)$ such that
\begin{eqnarray*}
\iota _ {{}_{\mathsf{0}}} &:=& \Id_ {{}_{\bbb (X,Y) }} \\
\iota _ {{}_{\mathsf{1}}} &:=& \chi _ {{}_{(X, LY)}}:  \ccc(LX, LY)\to \bbb (X, ULY) \\
\psi _ {{}_{\mathsf{2}}} &:=&  \chi _ {{}_{(X, LULY)}}: \ccc(LX, LULY)\to \bbb (X, ULULY) \\
\psi _ {{}_{\mathsf{3}}} &:=&  \chi _ {{}_{(X, LULULY)}}: \ccc(LX, LULULY)\to \bbb (X, ULULULY)
\end{eqnarray*}
\begin{equation*}
\begin{aligned}
\left(\iota _ {{}_{d}}\right)_{{}_f} &:= \eta_{{}_{f}}^{-1}\\
\left(\iota _ {{}_{d^0}}\right)_{{}_f} &:= \eta_ {{}_{U(f)\eta _ {{}_{X}}}}^{-1}\\
\left(\iota _ {{}_{d^1}}\right)_{{}_f} &:=  \uuuu _ {{}_{L(\eta _ {{}_{Y}})f }}\ast\id _ {{}_{\eta_ {{}_{X}} }} 
\end{aligned}
\qquad
\begin{aligned}
\left(\iota _ {{}_{\partial ^0}}\right)_{{}_f} &:=  \eta_ {{}_{U(f)\eta _ {{}_{X}}}}^{-1}\\
\left(\iota _ {{}_{\partial ^1}}\right) _ {{}_f} &:= \uuuu _ {{}_{L(\eta _ {{}_{ULY}})f }}\ast\id _ {{}_{\eta_ {{}_{X}} }}
\end{aligned}
\qquad
\begin{aligned}
\left( \iota _ {{}_{\partial ^2}}\right) _ {{}_f} &:= \uuuu _ {{}_{LUL(\eta _ {{}_{Y}})f }}\ast\id _ {{}_{\eta_ {{}_{X}} }} \\
\left(\iota _ {{}_{s^0}}\right)_{{}_f} &:=  \uuuu _ {{}_{\varepsilon _ {{}_{LY}} f }}\ast\id _ {{}_{\eta_ {{}_{X}} }}  
\end{aligned}
\end{equation*}
\normalsize
\end{proof}

We assume the existence of a biadjunction $J\dashv G $ in the commutative triangles below and study its counit and unit, provided that the biadjunctions $(E\dashv R, \rho , \mu , v,w), (L\dashv U, \eta , \varepsilon , s,t)$ induce the same pseudocomonad.
We start with the unit.

\begin{theo}[Unit]\label{unitunit}
Assume that $(E\dashv R, \rho , \mu , v,w), (L\dashv U, \eta , \varepsilon , s,t), (J\dashv G, \bar{\eta }, \bar{\varepsilon }, \bar{s}, \bar{t})$ are biadjunctions
such that the triangles
$$\xymatrix{  \aaa\ar[rr]^-{J}\ar[dr]_-{E}&&\bbb\ar[dl]^-{L} & \aaa\ar[rr]^-{J}&&\bbb\\
&\ccc & & &\ccc\ar[ul]^-{R}\ar[ur]_-{U}   & }$$
are commutative. If $(E\dashv R, \rho , \mu , v,w)$, $(L\dashv U, \eta , \varepsilon , s,t)$ induce the same pseudocomonad $\TTTTT $, then the following statements are equivalent:
\begin{enumerate}
\item The unit $\overline{\eta }: \Id _ {{}_\aaa}\longrightarrow GJ $ is a pseudonatural equivalence;
\item $E$ is pseudoprecomonadic;
\item The following $2$-functor is of effective descent for every pair of objects $(A,B)$ in $\aaa $
$$\widehat{\DDDD } ^A_B:\dot{\Delta }\to\CAT $$ 
\small
$$\xymatrix{\aaa (A, B) \ar[d]|-{E_{{}_{A,B}}}\\  \ccc (EA, EB)\ar@<-2.5ex>[rrr]_-{\ccc(EA, E(\rho _{{}_B} ))}\ar@<2.5ex>[rrr]^-{E_{{}_{{}_{A,REB}}}\circ\hspace{0.2em} \xi_ {{}_{{}_{(A,EB)}}}} &&& \ccc (EA, EREB )\ar[lll]|-{\ccc (EA,\mu _ {{}_{EB}})}
\ar@<-2.5 ex>[rrrr]_-{\ccc (EA, ERE(\rho _{{}_B} ))}\ar[rrrr]|-{\ccc (EA, E(\rho _ {{}_{REB}}))}\ar@<2.5ex>[rrrr]^-{E_{{}_{{}_{A,(RE)^2B}}}\circ \hspace{0.2em}\xi_ {{}_{{}_{(A,EREB)}}}} &&&& \ccc (EA, E(RE)^2B) }$$
\normalsize
in which $\xi : \ccc (E-,-)\simeq \aaa (-,R-) $ is the pseudonatural equivalence induced by the biadjunction $(E\dashv R, \rho , \mu , v,w)$ described in Remark \ref{addd}.
\item For each object $B$ of $\aaa $, the following diagram is of effective descent.
$$\widehat{\VVVV} _ {{}_{B}}: \dot{\Delta }\to\aaa $$
\small
\[\xymatrix{  B \ar[r]^-{\rho _ {{}_{B}}} & REB\ar@<-2.5ex>[rrr]_-{RE(\rho _{{}_B})}\ar@<2.5ex>[rrr]^-{\rho_ {{}_{REB}}} &&& REREY\ar[lll]|-{R(\mu _ {{}_{EB}})}
\ar@<-2.5 ex>[rrr]_-{RERE(\rho _ {{}_{B}})}\ar[rrr]|-{RE(\rho _ {{}_{REB}})}\ar@<2.5ex>[rrr]^-{\rho_ {{}_{REREB}}} &&& REREREB }\]
\normalsize
\end{enumerate}
\end{theo}

\begin{proof}
By Remark \ref{unitbiadjunction}, the unit $\bar{\eta } $ is a pseudonatural equivalence if and only if $J$ is locally an equivalence. Moreover, by the hypothesis and the universal property of the $2$-category of pseudocolagebras, we have the following diagram
\small
$$\xymatrix{  \aaa\ar[r]|{J} \ar@/^5ex/[rr]|{\widetilde{\KKKK } }\ar@{ { }}@/^3ex/[rr]|{\cong }\ar[rrd]|{E}&
\bbb\ar[r]|-{\KKKK }\ar[rd]|{L}&
\mathsf{Ps}\textrm{-}\TTTTT\textrm{-}\CoAlg\ar[d]|{\mathsf{L}}\\
&& \ccc	 }$$
\normalsize 
such that $\widetilde{\KKKK }, \KKKK $ are the comparison pseudofunctors. 

Since, by hypothesis, we know that $\KKKK $ is locally an equivalence, we conclude that $J $ is locally an equivalence if and only if  $\widetilde{\KKKK }$ is locally an equivalence. Thereby, to conclude, we just need to apply the characterizations of pseudoprecomonadic pseudofunctors: that is to say, Corollary \ref{preprepseudocomonadic} and Corollary \ref{CMEprecom}.
\end{proof}

Before studying the counit, for future references, we need the following result about the diagram $\VVVV _ Y: \dot{\Delta }\to \bbb $ in the context of biadjoint triangles.

\begin{lem}\label{k}
Let 
$$\xymatrix{  \aaa\ar[rr]^-{J}\ar[dr]_-{E}&&\bbb\ar[dl]^-{L} & \aaa\ar[rr]^-{J}&&\bbb\\
&\ccc  & & &\ccc\ar[ul]^-{R}\ar[ur]_-{U}   & }$$
be commutative triangles of pseudofunctors such that we have biadjunctions $(E\dashv R, \rho , \varepsilon , v,w)$ and $(L\dashv U, \eta , \varepsilon , s,t)$  inducing the same pseudocomonad
$\TTTTT = (\TTTTT, \varpi, \varepsilon, \Lambda, \delta , \mathsf{s}  ) $. 
We consider the diagram $\AAAA _Y : \Delta \to \aaa $. Then, for each object $Y$ of $\bbb $, there is a pseudonatural isomorphism
$$\zeta ^{{}^Y}: J\circ \AAAA _ Y\longrightarrow \VVVV _ Y\circ\j .$$
\end{lem}
\begin{proof}
Again, we have the same diagram of the proof of Theorem \ref{unitunit}.
In particular, for each object $Y$ of $\bbb$, there is an invertible $2$-cell 
$\yyyy _{{}_{Y}}: J(\rho _ {{}_{RLY}})\Rightarrow \eta _ {{}_{ULY}} $.
Thereby, we can define $\zeta ^{{}^Y} : J\circ \AAAA _ Y\longrightarrow \VVVV _ Y\circ\j $
such that the components $\zeta _ {{}_{\mathsf{1}}}, \zeta _ {{}_{\mathsf{2}}}, \zeta _ {{}_{\mathsf{3}}}$ are identity $1$-cells, the components
$\zeta ^{{}^Y} _ {{}_{d^1}} $, $\zeta ^{{}^Y} _ {{}_{s^0}}$, $\zeta ^{{}^Y}_ {{}_{\partial ^1}}$, $\zeta ^{{}^Y} _ {{}_{\partial ^2}} $
are identity $2$-cells and
\small
\begin{center}
\begin{equation*}
\begin{aligned}
\left(\zeta ^{{}^Y}_ {{}_{d^0}}\right) &:=  \yyyy _{{}_{Y}}\cdot J\left(\left((\rrrr\llll )_{{}_{LY}}^{-1}\cdot RL(t_{{}_{LY}})\right)\ast \id _ {{}_{\rho _ {{}_{RLY}}}}\right);
\end{aligned}
\begin{aligned}
\left(\zeta ^{{}^Y}_ {{}_{\partial ^0}}\right) &:=  \yyyy _{{}_{ULY}}\cdot J\left(\left((\rrrr\llll )_{{}_{LULY}}^{-1}\cdot RL(t_{{}_{LY}})\right)\ast \id _ {{}_{\rho _ {{}_{RLULY}}}}\right).
\end{aligned}
\end{equation*}
\end{center}
\normalsize
\end{proof}

\begin{theo}[Counit]\label{counit}
Let $(E\dashv R, \rho , \varepsilon , v,w)$ and $(L\dashv U, \eta , \varepsilon , s,t)$ be biadjunctions inducing the same pseudocomonad
$\TTTTT = (\TTTTT, \varpi, \varepsilon, \Lambda, \delta , \mathsf{s}  )$ such that the triangles of pseudofunctors
$$\xymatrix{  \aaa\ar[rr]^-{J}\ar[dr]_-{E}&&\bbb\ar[dl]^-{L} & \aaa\ar[rr]^-{J}&&\bbb\\
&\ccc  & & &\ccc\ar[ul]^-{R}\ar[ur]_-{U}   & }$$
commute. We assume that $(J\dashv G, \bar{\eta }, \bar{\varepsilon }, \bar{s}, \bar{t})$ is a biadjunction and $L$ is pseudoprecomonadic.
We consider the diagram $\AAAA _Y : \Delta \to \aaa $. Then $J \left\{\dot{\Delta } (\mathsf{0}, \j -), \AAAA _ Y\right\} _ {\bi } $ is the descent object of $J\circ\AAAA _ Y $ for every object $Y$ of $\bbb $ if and only if the counit $\overline{\varepsilon } : JG\longrightarrow\Id _ {{}_{\bbb}} $ is a pseudonatural equivalence.  
\end{theo}
\begin{proof}
Actually, this is a corollary of  Lemma \ref{k}, Corollary \ref{PMAIN} and Corollary \ref{CMEprecom}. More precisely, by Lemma \ref{k}, $ J\circ \AAAA _ Y\simeq \VVVV _ Y\circ\j $.
By Corollary \ref{CMEprecom}, since $L$ is pseudoprecomonadic, $\VVVV _ Y $ is of effective descent. Moreover, by the constructions of Theorem \ref{MAIN} (which proves Corollary \ref{PMAIN}), the counit is pointwise defined by the comparison $1$-cells
$$J\left\{ \dot{\Delta }(\mathsf{0}, \j -), \AAAA _ Y \right\}_{\bi }\to Y= \VVVV _ Y\mathsf{0}\simeq \left\{ \dot{\Delta }(\mathsf{0}, \j -), \VVVV _ Y\circ \j \right\} .$$ 
This completes the proof.
\end{proof}

\section{Pseudocomonadicity}\label{CMEE}

Similarly to the $1$-dimensional case, to prove the characterization of pseudocomonadic pseudofunctors employing the biadjoint triangle theorems, we need two results: Lemma \ref{lemassumed} and Proposition \ref{propassumed}, which are proved in \cite{CME} in Lemma 2.3 and Proposition 3.2 respectively.

We start with Lemma \ref{lemassumed}, which is a basic and known property of the diagram $\VVVV _ Y $. It follows from explicit calculations using the definition of descent objects: we give a sketch of the proof below.

\begin{lem}[\cite{CME}]\label{lemassumed}
Let $(L\dashv U, \eta , \varepsilon , s, t)$ be a biadjunction. For each object $Y$ of $\bbb $, the diagram $L\circ\VVVV _Y$ is of absolute effective descent.
\end{lem}
\begin{proof}
Trivially, given a pseudofunctor $\mathcal{F}: \ccc\to\mathfrak{Z} $, we can see $\mathcal{F}\circ L\circ\VVVV _Y$ as a $2$-functor, taking, if necessary, the obvious pseudonaturally equivalent version of $\mathcal{F}\circ L\circ\VVVV _ Y $. Then, for each objects $Z$ of $\mathfrak{Z} $, by Remark \ref{strictdescentstrict},  we can consider the strict descent object of the $2$-functor explicitly
$$\mathfrak{Z}( Z, \mathcal{F}\circ L\circ\VVVV _Y\circ \j -): \Delta\to\CAT $$
\small
\[\xymatrix{  \mathfrak{Z}(Z, \mathcal{F}LULY)\ar@<-2.5ex>[rrr]_-{\mathfrak{Z}(Z,\mathcal{F}LUL(\eta _{{}_Y}))}\ar@<2.5ex>[rrr]^-{\mathfrak{Z}(Z, \mathcal{F}L(\eta_ {{}_{ULY}}))} &&& \mathfrak{Z}(Z, \mathcal{F}LULULY)\ar[lll]|-{\mathfrak{Z}(Z, \mathcal{F}LU(\varepsilon _ {{}_{LY}}))}
\ar@<-2.5 ex>[rrr]_-{\mathfrak{Z}(Z, \mathcal{F}LULUL(\eta _ {{}_{Y}}))}\ar[rrr]|-{\mathfrak{Z}(Z,\mathcal{F}LUL(\eta _ {{}_{ULY}}))}\ar@<2.5ex>[rrr]^-{\mathfrak{Z}(Z, \mathcal{F}L(\eta_ {{}_{ULULY}}))} &&& \mathfrak{Z}(Z, \mathcal{F}LULULULY) }\]
\normalsize
Thereby, by straightforward calculations, taking Remark \ref{strictdescentstrict} into account,  we conclude that
\small
\begin{eqnarray*}
\mathfrak{Z}(Z, \mathcal{F}LY)&\to & \left\{ \dot{\Delta }(\mathsf{0}, \j -),\mathfrak{Z}( Z, \mathcal{F}\circ L\circ\VVVV _Y\circ \j -)\right\}\\
f&\mapsto & (\mathcal{F}L(\eta _ {{}_{Y}}) f, (\mathcal{F}L\eta) _ {{}_{\eta _ {{}_{Y}} }}\ast \id _ {{}_{f}} )\\
\mathfrak{m} &\mapsto & \id _ {{}_{{}_{\mathcal{F}L(\eta _ {{}_{Y}})}}}\ast \mathfrak{m} 
\end{eqnarray*}
\normalsize
gives an equivalence of categories (and it is the comparison functor). This completes the proof.
\end{proof}

\begin{prop}[\cite{CME}]\label{propassumed}
Let $\TTTTT = (\TTTTT, \varpi, \varepsilon, \Lambda, \delta , \mathsf{s}  ) $  be a pseudocomonad on $\ccc $. The forgetful pseudofunctor 
$L: \mathsf{Ps}\textrm{-}\TTTTT\textrm{-}\mathsf{CoAlg}\to  \ccc $ 
creates absolute effective descent diagrams.
\end{prop}

In this section, henceforth we work within the following setting (and notation): given a biadjunction $(E\dashv R, \rho , \mu , v,w)$, recall that, by Lemma \ref{basicstuff},
it induces  a biadjunction, herein denoted by $(L\dashv U, \eta , \varepsilon , s, t)$. We also get commutative triangles
\small
$$\xymatrix{\aaa \ar[r]^-{\KKKK}\ar[rd]_-{E}& \mathsf{Ps} \textrm{-} \TTTTT\textrm{-}\CoAlg \ar[d]^{L } &&\ccc \ar[r]^-{R}\ar[rd]_-{U}& \aaa \ar[d]^{\KKKK }\\
&\ccc &&& \mathsf{Ps} \textrm{-} \TTTTT\textrm{-}\CoAlg  }$$ 
\normalsize
in which, clearly, the biadjunctions $E\dashv R$, $L\dashv U $ induce the same pseudocomonad $\TTTTT$. In this context, if the  comparison pseudofunctor $\KKKK $ is a biequivalence, we say that $E$ is \textit{pseudocomonadic}. In other words, we say that $E $ is pseudocomonadic if there is a pseudofunctor $ G:\mathsf{Ps}\textrm{-}\TTTTT\textrm{-}\CoAlg\to\aaa $ such that $G\circ\KKKK\simeq \Id _ {{}_{\aaa }} $ and $\KKKK\circ G\simeq \Id _ {{}_{ \mathsf{Ps}\textrm{-}\TTTTT\textrm{-}\CoAlg}}$.

Of course, in the triangle above, the forgetful pseudofunctor $L$ is always pseudocomonadic. In particular, $L$ is always pseudoprecomonadic. Therefore the triangle satisfies the basic hypothesis of Corollary \ref{PMAIN}. 

Observe that, to verify the pseudocomonadicity of a left biadjoint pseudofunctor $L$, we can do it in three steps:
\begin{enumerate}
\item Verify whether $\KKKK $ has a right biadjoint via Corollary \ref{PMAIN};
\item If it does, the next step would be to verify whether the counit of the biadjunction $\KKKK\dashv G $ is a pseudonatural equivalence via Theorem \ref{counit};
\item The final step would be to verify whether the unit of the biadjunction $\KKKK\dashv G $ is a pseudonatural equivalences via Theorem \ref{unitunit}.
\end{enumerate}
These are precisely the steps used below.

\begin{theo}[Pseudocomonadicity~\cite{CME}]\label{pseudocomonadicityy}
A left biadjoint pseudofunctor $E: \aaa\to\ccc $ is pseudocomonadic if and only if it creates absolute effective descent diagrams.
\end{theo}
\begin{proof}
By Proposition \ref{propassumed}, pseudocomonadic pseudofunctors create absolute effective descent diagrams. Reciprocally, assume that $E$ creates absolute effective descent diagrams.

\begin{enumerate}
\item $\KKKK $ has a right biadjoint $G$:
\end{enumerate}

In this proof, we take a biadjunction $(E\dashv R, \rho , \varepsilon , v,w)$ and assume that $\TTTTT $ is its induced pseudocomonad. Also, we denote by $(L\dashv U, \eta , \varepsilon , s,t)$ the biadjunction induced by $\TTTTT $ (as described above). 

On one hand, by Lemma \ref{k} and Lemma \ref{lemassumed}, for each object $\mathsf{z} $ of $\mathsf{Ps}\textrm{-}\TTTTT\textrm{-}\CoAlg $, the diagram 
$\AAAA _\mathsf{z} : \Delta \to \aaa $   
is such that 
$E\circ\AAAA _\mathsf{z}\simeq L\circ\VVVV _ {\mathsf{z}}\circ\j $
is an absolute effective descent diagram, in which $\VVVV _ {\mathsf{z}}:\dot{\Delta } \to \mathsf{Ps} \textrm{-} \TTTTT\textrm{-}\CoAlg $ is induced by the biadjunction $L\dashv U$.

Therefore, since $E$ creates absolute effective diagrams, we conclude that there is an effective descent diagram $\BBBB _ {\mathsf{z}}$ such that $\AAAA _\mathsf{z} = \BBBB _ {\mathsf{z}}\circ\j $ and $E\circ \BBBB _ {\mathsf{z}}\simeq L\circ\VVVV _ {\mathsf{z}} $.
Thus, by Corollary \ref{PMAIN}, we conclude that $\KKKK $ has a right biadjoint $G$.
\begin{enumerate}\setcounter{enumi}{1}
\item The counit of the biadjunction $\KKKK\dashv G $ is a pseudonatural equivalence:
\end{enumerate}
Since $L\circ\KKKK\circ \BBBB _ {\mathsf{z}}= E\circ \BBBB _ {\mathsf{z}}\simeq L\circ\VVVV _ {\mathsf{z}} $ is of absolute effective descent and $L$ creates absolute effective descent diagrams, we conclude that  $\KKKK\circ \BBBB_\mathsf{z} $ is of effective descent. By Theorem \ref{counit}, it completes this second step.

\begin{enumerate}\setcounter{enumi}{2}
\item The unit of the biadjunction $\KKKK\dashv G $ is a pseudonatural equivalence:
\end{enumerate}
By Lemma \ref{lemassumed}, for every object $A$ of $\aaa $, 
$E\circ \widehat{\VVVV} _A : \dot{\Delta}\to \ccc $
is of absolute effective descent, in which $\widehat{\VVVV} _A$ is induced by the biadjunction $E\dashv R $. Since $E$ creates absolute effective descent diagrams, we get that $\widehat{\VVVV} _A $ is of effective descent. Therefore, by Corollary \ref{CMEprecom}, $E$ is pseudoprecomonadic. By Theorem \ref{unitunit}, it completes the proof of the final step.
\end{proof}

As a consequence of Theorem \ref{pseudocomonadicityy}, 
within the setting of Theorem \ref{MAIN}, if $J$ has a right biadjoint and $ E$ is pseudocomonadic, 
then $J$ is pseudocomonadic as well. Furthermore, it is worth to point out that the second step of 
the proof of Theorem \ref{pseudocomonadicityy} follows directly from the fact that $E$  
preserves the effective descent diagrams $\BBBB _ {{}_{\mathsf{z}}}$ and from the 
pseudocomonadicity of $L$. More precisely, as direct consequence of Lemma \ref{lemassumed}, Theorem \ref{counit} and Proposition \ref{propassumed}, we get:

\begin{coro}[Counit]\label{generalizedcoherence}
Let $(E\dashv R, \rho , \mu , v,w), (L\dashv U, \eta , \varepsilon , s,t)$ be biadjunctions inducing the same pseudocomonad $\TTTTT = (\TTTTT, \varpi, \varepsilon, \Lambda, \delta , \mathsf{s}  ) $ such that
$$\xymatrix{  \aaa\ar[rr]^-{J}\ar[dr]_-{E}&&\bbb\ar[dl]^-{L}\\
&\ccc & }$$
\normalsize
commutes. Assume that $L$ is pseudocomonadic, $J\circ R=U$ and $(J, G, \overline{\varepsilon}, \overline{\eta}, \overline{s}, \overline{t})$ is a biadjunction. The counit
$\overline{\varepsilon }: JG\longrightarrow \Id _ {{}_{\bbb }}$
is a pseudonatural equivalence if and only if, for every object $Y$ of $\bbb$, $E$ preserves the descent object of 
$\AAAA _Y : \Delta \to \aaa $.
\end{coro}
\begin{proof}
By Corollary \ref{PMAIN}, since $J$ is left biadjoint, for each object $Y$ of $\bbb $, there is an effective descent diagram $\BBBB _ Y : \dot{\Delta }\to\aaa $ such that $\BBBB _ Y\circ\j\simeq \AAAA _ Y $.
By the commutativity of the triangles $L\circ J=E$ and $J\circ R=U $, since $(E\dashv R, \rho , \mu ) $ and $(L\dashv U, \eta , \varepsilon , s, t)$ induce the same pseudocomonad, our setting satisfies the hypotheses of Lemma \ref{k}. Thus, for each object $Y$ of $\bbb $, there is a pseudonatural equivalence 
$$J\circ\BBBB _ Y\circ\j \simeq J\circ \AAAA _ Y\simeq\VVVV _Y \circ \j .$$

By Theorem \ref{counit}, to complete this proof, it is enough to show that $J\circ \BBBB _ Y $ is of effective descent if and only if $E\circ \BBBB _ Y $ is of effective descent.

Firstly, we assume that $J\circ \BBBB _ Y $ is of effective descent. In this case, since $J\circ\BBBB _ Y\circ\j \simeq \VVVV _Y \circ \j $ and $\VVVV _ Y $ is of effective descent, we conclude that $\VVVV _ Y\simeq J\circ\BBBB _ Y$. Thus, by Lemma \ref{lemassumed}, 
$$L\circ \VVVV _ Y\simeq L\circ J\circ \BBBB _ Y = E\circ \BBBB _ Y $$
is, in particular, of effective descent.

Reciprocally, we assume that $E\circ \BBBB _ Y $ is of effective descent. Again, since $E\circ\BBBB _ Y\circ\j \simeq L\circ\VVVV _Y \circ \j $ and $L\circ\VVVV _Y$ is of absolute effective descent, we conclude that $E\circ\BBBB _ Y\simeq L\circ \VVVV _ Y $ is of absolute effective descent. Therefore, since $L$ is pseudocomonadic, by Proposition \ref{propassumed}, we conclude that $J\circ\BBBB _ Y $ is of effective descent.

\end{proof}

\section{Coherence}\label{coherence}
A $2$-(co)monadic approach to \textit{coherence} consists of studying the inclusion of the $2$-category of strict (co)algebras into the $2$-category of pseudo(co)algebras of a given $2$-(co)monad to get general coherence results~\cite{Power, SLACK2002, Power89}. More precisely, one is interested, firstly, to understand whether the inclusion of the $2$-category of strict coalgebras into the $2$-category of pseudocoalgebras has a right $2$-adjoint $G$ (what is called a ``coherence theorem of the first type'' in \cite{SLACK2002}). Secondly, if there is such a right $2$-adjoint,  one is interested in investigating whether every pseudocoalgebra $\mathsf{z} $ is equivalent to the strict replacement $G(\mathsf{z})$ (what is called  a ``coherence theorem of the second type'' in \cite{SLACK2002}).

We fix the notation of this section as follows: we have a $2$-comonad $\TTTTT = (\TTTTT, \varpi, \varepsilon ) $ on a $2$-category $\ccc$. We denote by  $\TTTTT$-$\CoAlg  _{{}_{\mathsf{s}}} $ the $2$-category of strict coalgebras, strict morphisms and $\TTTTT$-transformations, that is to say, the usual $\CAT$-enriched category of coalgebras of the $\CAT $-comonad $\TTTTT $. The $2$-adjunction $E\dashv R :\TTTTT\textrm{-}\CoAlg  _{{}_{\mathsf{s}}}\to\ccc $ induces the Eilenberg-Moore factorization w.r.t. the pseudocoalgebras:
$$\xymatrix{\TTTTT\textrm{-} \CoAlg  _{{}_{\mathsf{s}}}\ar[r]^-{J}\ar[rd]_-{E}& \mathsf{Ps} \textrm{-} \TTTTT\textrm{-} \CoAlg \ar[d]^-{L}\\
&\ccc }$$
in which $J: \TTTTT\textrm{-} \CoAlg  _{{}_{\mathsf{s}}}\to\mathsf{Ps}\textrm{-}\TTTTT\textrm{-}\CoAlg $ is the usual inclusion.

Firstly, Corollary \ref{PMAIN} gives, in particular, necessary and sufficient conditions for which a $2$-comonad satisfies the ``coherence theorem of the first type'' and a weaker version of it, that is to say, it also studies when $J$ has a right biadjoint $G$.
Secondly, Corollary \ref{generalizedcoherence} gives necessary and sufficient conditions for getting a stronger version of the ``coherence theorem of the second type'', that is to say, it studies when the counit of the obtained biadjunction/$2$-adjunction is a pseudonatural equivalence. 

\begin{coro}[Coherence Theorem]\label{1}
Let $\TTTTT = (\TTTTT, \varpi, \varepsilon ) $ be a $2$-comonad on a $2$-category $\ccc $. It induces a $2$-adjunction $(E\dashv R, \rho , \varepsilon )$ and a biadjunction $(L\dashv U, \eta , \varepsilon , s,t)$ such that
\small
$$\xymatrix{\TTTTT\textrm{-} \CoAlg  _{{}_{\mathsf{s}}} \ar[r]^-{J}\ar[rd]_-{E}& \mathsf{Ps} \textrm{-} \TTTTT\textrm{-} \CoAlg \ar[d]^-{L}\\
&\ccc }$$
\normalsize
commutes.
The inclusion 
$J: \TTTTT\textrm{-}\CoAlg  _{{}_{\mathsf{s}}}\to\mathsf{Ps}\textrm{-}\TTTTT\textrm{-}\CoAlg $
has a right biadjoint if and only if $\TTTTT\textrm{-}\CoAlg  _{{}_{\mathsf{s}}}$ has the descent object of
\begin{equation*}\tag{$\AAAA _\mathsf{z}$}  
\xymatrix{  RL\mathsf{z}\ar@<-2.5ex>[rrrr]_-{RL(\eta _{{}_\mathsf{z}})}\ar@<2.5ex>[rrrr]^-{\rho _ {{}_{RL\mathsf{z}}}} &&&&  R\TTTTT L\mathsf{z} \ar[llll]|-{R(\varepsilon _{{}_{L\mathsf{z}}})}
\ar@<-2.5 ex>[rrrr]_-{R\TTTTT L(\eta_{{}_{\mathsf{z}}})}\ar[rrrr]|-{RL(\eta_{{}_{UL\mathsf{z} }})}\ar@<2.5ex>[rrrr]^-{ \rho _{{}_{R\TTTTT L\mathsf{z}}}} &&&& R\TTTTT ^2 L\mathsf{z} }
\end{equation*}
for every pseudocoalgebra $\mathsf{z}$ of $\mathsf{Ps}\textrm{-}\TTTTT\textrm{-}\CoAlg $. In this case, $J$ is left biadjoint to $G$, given by
$G\mathsf{z}:= \left\{ \dot{\Delta } (\mathsf{0}, \j - ), \AAAA _\mathsf{z}\right\} _{\bi }$.
Moreover, assuming the existence of the biadjunction $(J\dashv G, \overline{\varepsilon}, \overline{\eta}, \overline{s}, \overline{t})$, the counit
$\overline{\varepsilon }: JG\longrightarrow \id _ {{}_{\mathsf{Ps}\textrm{-}\TTTTT\textrm{-}\CoAlg }}$
is a pseudonatural equivalence if and only if $E$ preserves the descent object of $\AAAA _ \mathsf{z} $ for every pseudocoalgebra $\mathsf{z}$.

Furthermore, $J$ has a genuine right $2$-adjoint $G$ if and only if $\TTTTT\textrm{-}\CoAlg  _{{}_{\mathsf{s}}}$ admits the strict descent object of $\AAAA _\mathsf{z}$ for every $\TTTTT $-pseudocoalgebra $\mathsf{z}$. In this case, the right $2$-adjoint is given by $\displaystyle G\mathsf{z}:= \left\{ \dot{\Delta } (\mathsf{0}, \j - ), \AAAA _\mathsf{z}\right\}$.

\end{coro}
\begin{proof}
Since $(E\dashv R, \rho , \varepsilon )$ and $(L\dashv U, \eta , \varepsilon , s,t)$ induce the same pseudocomonad and $(\eta J ) = (J\rho ) $ is a $2$-natural transformation, it is enough to apply Corollary \ref{generalizedcoherence} and Corollary \ref{PMAIN} to the triangle $L\circ J = E$.
\end{proof}

We say that a $2$-comonad $\TTTTT $ satisfies the \textit{main coherence theorem} if there is a right $2$-adjoint $\mathsf{Ps}\textrm{-}\TTTTT\textrm{-}\CoAlg\to \TTTTT\textrm{-}\CoAlg _{{}_{\mathsf{s}}} $ to the inclusion and the counit of such $2$-adjunction is a pseudonatural equivalence.
 
To get the original statement of \cite{SLACK2002}, we have to employ the following well known result (which is a consequence of a more general result on  enriched comonads):

\textit{Let $\TTTTT$ be a $2$-comonad on $\ccc $. The forgetful $2$-functor $\TTTTT\textrm{-}\CoAlg _{{}_{\mathsf{s}}}\to\ccc $ creates all those strict descent objects which exist in $\ccc $ and are preserved by $\TTTTT $ and $\TTTTT ^2 $.}

Employing this result and Corollary \ref{1}, we prove Theorem 3.2 and Theorem 4.4 of \cite{SLACK2002}. For instance, we get:

\begin{coro}[\cite{SLACK2002}]
Let $\TTTTT $ be a $2$-comonad on a $2$-category $\ccc $. If $\ccc $ has and $\TTTTT $ preserves strict descent objects, then $\TTTTT $ satisfies the main coherence theorem.
\end{coro}

\section{On lifting biadjunctions}\label{App}
One of the most elementary corollaries of the adjoint triangle theorem~\cite{Dubuc} is about lifting adjunctions to adjunctions between the Eilenberg-Moore categories. In our case, let  $\TTTTT:\aaa \to \aaa $ and $\SSSSS : \ccc\to\ccc $ be $2$-comonads (with omitted comultiplications and counits), if 
$$\xymatrix{  \TTTTT\textrm{-}\CoAlg  _{{}_{\mathsf{s}}}\ar[d]_{\widehat{L}}\ar[r]^J & \SSSSS\textrm{-}\CoAlg  _{{}_{\mathsf{s}}}\ar[d]^{L}\\ 
\aaa\ar[r]_{E}& \ccc  }$$
is a commutative diagram, such that $E$ has a right $2$-adjoint $R$, then Proposition \ref{D1} gives necessary and sufficient conditions to construct a right $2$-adjoint to $J$. Also, of course, as a  consequence of Corollary \ref{PMAIN}, we have the analogous version for pseudocomonads.

\begin{coro}\label{lifting}
Let $\TTTTT :\aaa\to\aaa $ and $\SSSSS :\ccc\to \ccc $ be pseudocomonads. If the diagram 
$$\xymatrix{  \mathsf{Ps} \textrm{-}\TTTTT \textrm{-} \CoAlg \ar[d]_{\widehat{L}}\ar[r]^J & \mathsf{Ps} \textrm{-} \SSSSS \textrm{-} \CoAlg \ar[d]^{L}\\ 
\aaa\ar[r]_{E}& \ccc  }$$
commutes and $E$ has a right biadjoint, then $J$ has a right biadjoint provided that $\mathsf{Ps}\textrm{-}\TTTTT\textrm{-} \CoAlg $ has descent objects.
\end{coro}

Recall that $\mathsf{Ps}\textrm{-}\TTTTT\textrm{-} \CoAlg $ has descent objects if $\aaa $ has and $\TTTTT $ preserves descent objects. Therefore the pseudofunctor $J$ of the last result has a right biadjoint in this case.

\subsection{On pseudo-Kan extensions}
One simple application of Corollary \ref{lifting} is about pseudo-Kan extensions. In the tricategory $2$-$\CAT $, the natural notion of Kan extension is that of pseudo-Kan extension. More precisely, a \textit{right pseudo-Kan extension of a pseudofunctor $\mathcal{D}: \sss\to \aaa $ along a pseudofunctor $\h : \sss\to\dot{\sss} $, denoted by $\Ps \textrm{-}\Ran _ {\h } \mathcal{D} $, is (if it exists) a birepresentation of the pseudofunctor 
$\WWWW\mapsto [\sss , \aaa ]_ {PS}(\WWWW\circ \h , \mathcal{D})$. } Recall that birepresentations are unique up to equivalence and, therefore, right pseudo-Kan extensions are unique up to pseudonatural equivalence.

Assuming that $\h:\sss\to\dot{\sss } $ is a pseudofunctor between small $2$-categories, in the setting described above, the following are natural problems on pseudo-Kan extensions: (1) investigating the left biadjointness of the pseudofunctor $\WWWW\to \WWWW\circ\h $, namely, investigating whether all right pseudo-Kan extensions along $\h $ exist; (2) understanding pointwise pseudo-Kan extensions (that is to say, proving the existence of right pseudo-Kan extensions provided that $\aaa $ has all bilimits).

It is shown in \cite{Power} that, if $\sss _ 0 $ denotes the discrete $2$-category of the objects of $\sss $, the restriction $[\sss , \aaa ]\to [\sss _ 0 , \aaa ] $ is $2$-comonadic, provided that $[\sss , \aaa ]\to [\sss _ 0 , \aaa ] $ has a right $2$-adjoint $\Ran _ {\sss\to\sss _ 0 } $. 
It is also shown there that the $2$-category of pseudocoalgebras of the induced $2$-comonad is $ [\sss  , \aaa ] _ {PS} $. It actually works more generally:  
$[\sss , \aaa ] _{PS}\to [\sss _ 0 , \aaa ] _ {PS} = [\sss _ 0 , \aaa ]$ is pseudocomonadic whenever there is a right biadjoint  $\Ps \textrm{-}\Ran _ {\sss _ 0\to\sss } : [\sss _ 0, \aaa ] _{PS}\to [\sss  , \aaa ] _{PS} $ 
 because existing bilimits of $\aaa $ are constructed objectwise in $[\sss , \aaa ] _{PS}$ (and, therefore, the hypotheses of the pseudocomonadicity theorem~\cite{CME} are satisfied). Thus, we get the following commutative square:
$$\xymatrix{ [\dot{\sss }, \aaa ] _{PS}\ar[d]\ar[r]^{[\h , \aaa ] _{PS} } & [\sss , \aaa ] _{PS}\ar[d]\\ 
[\dot{\sss _ 0 }, \aaa ] \ar[r]_{[\h , \aaa ] _{PS}}& [\sss  _ 0, \aaa ]  }$$  
Thereby, Corollary \ref{lifting} gives a way to study pseudo-Kan extensions, even in the absence of strict $2$-limits. That is to say, on  one hand, if the $2$-category $\aaa $ is complete, our results give pseudo-Kan extensions as descent objects of strict $2$-limits. On the other hand, in the absence of strict $2$-limits and, in particular, assuming that $\aaa $ is bicategorically complete,  we can construct the following pseudo-Kan extensions:
\begin{eqnarray*}
\Ps \textrm{-}\Ran _ {{}_{\sss _ 0 \to \dot{\sss }_0 } } : & [\sss _ 0, \aaa ] &\to  [\dot{\sss }_0, \aaa ] _{PS}\\
& \mathcal{D} & \mapsto  \Ps \textrm{-}\Ran _ {{}_{\sss _ 0 \to \dot{\sss }_0 } } \mathcal{D}: \left( x\mapsto \prod _{\h(a)=x } \mathcal{D}a\right)\\
&&\\
\Ps \textrm{-}\Ran_ {{}_{\dot{\sss } _ 0 \to \dot{\sss }} } : & [\dot{\sss } _ 0, \aaa ] &\to  [\dot{\sss }, \aaa ] _{PS}\\
& \mathcal{D} & \mapsto  \Ps \textrm{-}\Ran_ {{}_{\dot{\sss } _ 0 \to \dot{\sss }} } \mathcal{D} : \left( x\mapsto \prod _{y\in \dot{\sss } _ 0 } \dot{\sss } (x, y)\pitchfork \mathcal{D} y\right)\\
&&\\
\Ps \textrm{-}\Ran _ {{}_{\sss _ 0 \to \sss } } : & [\sss _ 0, \aaa ] &\to  [\sss , \aaa ] _{PS}\\
& \mathcal{D} & \mapsto  \Ps \textrm{-}\Ran _ {{}_{\sss _ 0 \to \sss } } \mathcal{D}: \left( a\mapsto \prod _{b\in \sss _ 0 } \sss (a, b)\pitchfork \mathcal{D} b\right)
\end{eqnarray*}
in which $\prod $ and $\pitchfork $ denote the bilimit versions of the product and cotensor product, respectively. Thereby, by Corollary \ref{lifting}, the pseudo-Kan extension $\Ps \textrm{-}\Ran _{\h } $ can be constructed pointwise as descent objects of a diagram obtained from the pseudo-Kan extensions above. Namely, $\Ps \textrm{-}\Ran _ {\h }\mathcal{D}x $ is the descent object of a diagram
\small
\[\xymatrix{  \mathfrak{a}_ 0\ar@<1.5ex>[rr]\ar@<-1.5ex>[rr] 
&& \mathfrak{a}_1\ar[ll]\ar@<1.5 ex>[rr]\ar[rr]\ar@<-1.5ex>[rr] 
&& \mathfrak{a}_2 }\]
\normalsize
in which, by Theorem \ref{MAIN} and the last observations, 
\small
\begin{eqnarray*}
\mathfrak{a}_ 0 &=& \displaystyle \prod _{y\in \dot{\sss  _ 0 }} \left(\dot{\sss } (x, y) \pitchfork \prod _{\h (a)=y } \mathcal{D}a \right)\simeq  \displaystyle\prod _{a\in \sss _ 0 } \left( \dot{\sss}(x,\h (a))\pitchfork \mathcal{D}a\right) \\
\\
\mathfrak{a}_ 1 & = & \left(\dot{\sss } (x, y) \pitchfork \prod _{\h (a)=y } \left( \prod _{b\in \sss _ 0 } \sss (a, b)\pitchfork \mathcal{D} b\right) \right)\\
                 &\simeq & \displaystyle\prod _{a\in \sss  _ 0 } \left( \dot{ \sss }(x,\h (a))\pitchfork  \left( \prod _{b\in \sss _ 0 } \sss (a, b)\pitchfork \mathcal{D} b\right)\right)\\
                 &\simeq & \displaystyle\prod_ {(a,b)\in \sss  _ 0\times \sss _ 0 } \left(\left(\sss (a,b)\times \dot{\sss }(x,\h (a))\right) \pitchfork \mathcal{D}b\right)\\
\\
\mathfrak{a}_ 2 &\simeq & \displaystyle \prod _ {(a, b, c)\in \sss  _ 0\times \sss _ 0\times \sss _ 0} \left( \left( \sss (b, c)\times \sss (a, b)\times \dot{\sss } (x, \h (a)) \right)\pitchfork \mathcal{D}c\right) 
\end{eqnarray*}
\normalsize
This implies that, indeed, if $\aaa $ is bicategorically complete, then $\Ps \textrm{-} \Ran _ {\h } \mathcal{D} $ exists and, once we assume the results of \cite{ RS87} related to the construction of weighted bilimits via descent objects, we conclude that:

\begin{prop}[Pointwise pseudo-Kan extension]\label{point}
Let $\sss , \dot{\sss } $ be small $2$-categories and $\aaa $ be a bicategorically complete $2$-category. If $\h : \sss\to\dot{\sss } $ is a pseudofunctor, then $$\Ps \textrm{-} \Ran _ {\h } \mathcal{D} x = \left\{ \dot{\sss } (x, \h -), \mathcal{D}\right\} _ {\bi } $$ 
\end{prop}

\begin{coro}
If $\mathcal{A}:\Delta\to\aaa $ is a pseudofunctor and $\aaa $ has the descent object of $\mathcal{A}$, then $\Ps \textrm{-} \Ran _ {\j } \mathcal{A}\mathsf{0}$ is the descent object of $\mathcal{A} $.
\end{coro}

Moreover, by the bicategorical Yoneda Lemma, we get: 

\begin{coro}
If $\h : \sss\to\dot{\sss } $ is locally an equivalence and there is a biadjunction
$[\h , \aaa]\dashv \Ps\textrm{-}\Ran _ {\h} $,
its counit is a pseudonatural equivalence.
\end{coro}

Finally, let $\aaa $ be a $2$-category with all descent objects and $\TTTTT$ be a pseudocomonad on $\aaa $. Recall that, if $\TTTTT$ preserves all effective descent diagrams, $\mathsf{Ps} \textrm{-}\TTTTT\textrm{-} \CoAlg $
 has all descent objects. Therefore, if $\h : \sss\to\dot{\sss } $ is a pseudofunctor, in this setting, the commutative diagram below satisfies the hypotheses of Corollary \ref{lifting} (and, thereby, it can be used to lift pseudo-Kan extensions to pseudocoalgebras).
\small
$$\xymatrix{ [\dot{\sss } , \mathsf{Ps} \textrm{-}\TTTTT\textrm{-} \CoAlg ] _{PS}\ar[d]\ar[r] & [\sss  , \mathsf{Ps} \textrm{-}\TTTTT\textrm{-} \CoAlg] _{PS}\ar[d]\\ 
[\dot{\sss }, \aaa ] _{PS}\ar[r]& [\sss  , \aaa ] _{PS} }$$
\normalsize  

\begin{rem}
Assume that $\h : \sss\to\dot{\sss } $ is a pseudofunctor, in which $\sss , \dot{\sss } $ are small $2$-categories.
There is another way of proving Proposition \ref{point}. Firstly, we define the bilimit version of \textit{end}. That is to say, if $T: \sss\times\sss ^{\op }\to\CAT $ is a pseudofunctor, we define 
\small
$$ \displaystyle\int _ {\sss } T : = \left[ \aaa\times \aaa ^{\op }, \CAT\right] _{PS} \left( \aaa (-,-), T\right) $$
\normalsize
From this definition, it follows Fubini's theorem (up to equivalence). And, if $\mathcal{B}, \mathcal{D}: \sss\to \aaa $ are pseudofunctors, the following equivalence holds:
\small
$$ \displaystyle\int _ {\sss } \aaa (\mathcal{B}a, \mathcal{D}a)  \simeq \left[ \sss, \aaa\right] _{PS} \left( \mathcal{B}, \mathcal{D} \right) $$
\normalsize
Therefore, if $\h : \sss\to\dot{\sss } $ is a pseudofunctor and we define $\Ps\Ran _ {\h } \mathcal{D}x = \left\{ \dot{\sss } (x, \h -), \mathcal{D}\right\} _{\bi } $, we have the pseudonatural equivalences (analogous to the enriched case \cite{Kelly}) 
\small
\begin{eqnarray*}
\left[\dot{\sss }, \aaa \right] _ {PS}(\WWWW , \Ps\Ran _ {\h } \mathcal{D}) &\simeq & \int _ {\dot{\sss }}\aaa (\WWWW x, \Ps\Ran _ \h \mathcal{D}x)\\
                                                               &\simeq & \int _{\dot{\sss }}\aaa (\WWWW x, \left\{ \dot{\sss } (x, \h -), \mathcal{D}\right\} _{\bi })  \\
                                                             & \simeq & \int _ {\dot{\sss } }\left[ \sss , \CAT\right]_ {PS}(\dot{\sss } (x, \h -), \aaa (\WWWW x, \mathcal{D}-))\\
																														& \simeq & \int _ {\dot{\sss }} \int _ {\sss } \CAT (\dot{\sss } (x, \h (a)), \aaa (\WWWW x, \mathcal{D}a))\\
																														& \simeq & \int _ {\sss } \int _ {\dot{\sss }} \CAT (\dot{\sss } (x, \h (a)), \aaa (\WWWW x, \mathcal{D}a))\\
                                                             &\simeq & \int _ {\sss } \left[\dot{\sss }^{\op }, \CAT\right] _ {PS} (\dot{\sss } (-, \h (a)), \aaa (\WWWW -, \mathcal{D}a))\\
																														&\simeq &   \int _ {\sss } \aaa (\WWWW\circ\h (a), \mathcal{D}a)\\
                                                            &\simeq & \left[ \sss , \aaa\right]_ {PS} (\WWWW\circ \h , \mathcal{D})
\end{eqnarray*}
\normalsize
This completes the proof that if the pointwise right pseudo-Kan extension $\Ps\Ran _{\h } $ exists, it is a right pseudo-Kan extension. Within this setting and assuming this result, the original argument used to prove Proposition \ref{point} using biadjoint triangles gets the construction via descent objects of weighted bilimits originally given in \cite{RS87}.  
\end{rem}

 \end{document}